\documentclass[12pt]{amsart}
\usepackage{amsmath,amssymb,amscd,verbatim,xspace,amsthm,tensor}
\usepackage{graphicx,enumerate,tikz,tikz-cd,appendix}
\usepackage{latexsym, epsfig, color}
\usepackage[hidelinks]{hyperref}
\usepackage{fullpage}
\usepackage{pdfsync}
\usepackage[OT2,T1]{fontenc}
\usepackage{amsfonts}
\usepackage{mathrsfs}

\numberwithin{equation}{section}

\theoremstyle{plain}
\newtheorem{theorem}{Theorem}[section]
\newtheorem{proposition}[theorem]{Proposition}
\newtheorem{corollary}[theorem]{Corollary}
\newtheorem{lemma}[theorem]{Lemma}

\newtheorem{hypothesis}[theorem]{Hypothesis}

\theoremstyle{definition}
\newtheorem{definition}[theorem]{Definition}

\newtheorem{example}[theorem]{Example}

\theoremstyle{remark}
\newtheorem*{remark}{Remark}

\newcommand\cut[1]{}
\newcommand\further[1]{}
\begin{document}
	
	\title[Distribution of the Bad Part of Class Groups]{Distribution of the Bad Part of Class Groups}
	\date{}
	\author{Weitong Wang}
	\address{Shing-Tung Yau Center of Southeast University, 15th floor Yifu Architecture Building, 2 Sipailou, Nanjing, Jiangsu Province 210096 China}
	\email{wangweitong@seu.edu}	
	
	\begin{abstract}
		Let $G$ be a finite group.
		The Cohen-Lenstra-Martinet Heuristics give a prediction of the distribution of $\operatorname{Cl}_K[p^\infty]$ when $K$ runs over $G$-fields and $p\nmid\lvert G\rvert$.
		In this paper, we prove several results on the distribution of ideal class groups for some $p\mid\lvert G\rvert$,
		and show that the behaviour is qualitatively different than what is predicted by the heuristics when $p\nmid\lvert G\rvert$.
		We do this by using genus theory and the invariant part of the class group to investigate the algebraic structure of the class group.
		For general number fields, our result is conditional on a natural conjecture on counting fields.
		For abelian or $D_4$-fields, our result is unconditional.
	\end{abstract}
	
	\keywords{Cohen-Lenstra heuristics, distribution of class groups, field counting}
	
	
	
	\maketitle

\section{Introduction}\label{section:intro}
Let's start with the basic notions of this paper.
First of all, there are multiple ways to describe field extensions, we give the following definitions to make terms like ``the set of all non-Galois cubic number fields'' precise. 
\begin{definition}\label{def:Gamma-fields}
	Let $k$ be a number field, and let $G$ be a finite group.
	By a \emph{$G$-extension of $k$}, we mean an isomorphism class of pairs $(K,\psi)$, where $K$ is a Galois extension of $k$, and $\psi:G(K/k)\cong G$ is an isomorphism.  
	An isomorphism of pairs $(\alpha,m_\alpha):(K,\psi)\to(K',\psi')$ is an isomorphism $\alpha: K\to K'$ such that the map $m_\alpha: G(K/k)\to G(K'/k)$ sending $\sigma$ to $\alpha\circ\sigma\circ\alpha^{-1}$ satisfies $\psi'\circ m_\alpha=\psi$.  
	We sometimes leave the $\psi$ implicit.
	If the base field $k=\mathbb{Q}$, we just call $G$-extensions of $\mathbb{Q}$ \emph{$G$-fields.}
\end{definition}
\begin{definition}\label{def:set of fields}
	Let $G\subseteq S_n$ whose action on $\{1,2,\dots,n\}$ is transitive, i.e., a transitive permutation group.
	Let $k$ be a number field.
	Let ${\mathcal{S}}(G,k)$ be the set of pairs $(K,\psi)$ such that the Galois closure $(\hat{K},\psi)$ is a $G$-extension of $k$ and that $K=\hat{K}^{\operatorname{Stab}(1)}$, where $\operatorname{Stab}(1)\subseteq G$ is the image of stabilizer of $1$.
	In other words, $\psi$ defines the Galois action of $G$ on the $k$-embeddings $K\to\mathbb{C}$ and $[K:k]=n$.
	If the base field $k=\mathbb{Q}$, then we just omit it and write ${\mathcal{S}}(G):={\mathcal{S}}(G,\mathbb{Q})$.
\end{definition}
Then, we give the notation of counting number fields.
\begin{definition}\label{def:counting number fields}
	Let $\mathcal{S}$ be a set of fields equipped with a fixed counting function of number fields $C:\mathcal{S}\to\mathbb{R}^+$ (e.g. discriminant or product of ramified primes).
	Define
	\begin{equation*}
		N_{\mathcal{S},C}(X):=\#\{K\in\mathcal{S}\mid C(K)<X\}.
	\end{equation*}
\end{definition}
Now we can define the notations of probability and moments.
We in particular care about the $p$-rank of class groups in the sense of statistics.
Let $A$ be a finite abelian group, and let $p$ be a rational prime.
Define the $p$-rank of $A$, denoted by $\operatorname{rk}_pA$, as the largest number $r$ so that there exists some injective group homomorphism $C_p^r\to A$.
Let $\mathcal{S}$ be a set of fields with a counting function $C$.
For each non-negative integer $r$, define
\begin{equation*}
	\mathbb{P}_{\mathcal{S},C}(\operatorname{rk}_p\operatorname{Cl}_K\leq r):=\lim_{X\to\infty}\frac{\#\{K\in\mathcal{S}\mid C(K)<X\text{ and }\operatorname{rk}_p\operatorname{Cl}_K\leq r\}}{N_{\mathcal{S},C}(X)},
\end{equation*}
and call it the probability of $\operatorname{rk}_p\operatorname{Cl}_K\leq r$.
Define also the $A$-moment of $\operatorname{Cl}_K$ to be
\begin{equation*}
	\mathbb{E}_{\mathcal{S},C}(\lvert\operatorname{Hom}(\operatorname{Cl}_K,A)\rvert):=
	\lim_{X\to\infty}
	\frac{
		\sum_{
			\substack{
				K\in\mathcal{S}\\
				C({K})<X
				}
			}
		\lvert\operatorname{Hom}(\operatorname{Cl}_{K},A)\rvert
		}{
		N_{\mathcal{S},C}(X)
	}.
\end{equation*}
When $\mathcal{S}$ is the set of Galois $G$-extensions with a fixed decomposition group at infinity, the class groups are equipped with an action from $G$.
If $p\nmid\lvert G\rvert$ is a rational prime, the \emph{Cohen-Lenstra-Martinet Heuristics} gives the prediction of the distribution of the $\mathbb{Z}_p[G]$-module $\operatorname{Cl}_K[p^\infty]:=\operatorname{Cl}_K\otimes\mathbb{Z}_p$ and the $M$-moments where $M$ is a finite $\mathbb{Z}_p[G]$-module.
See Cohen and Martinet~\cite{CM90}.
The method of Cohen and Martinet could be applied to the non-Galois cases.
See W. and Wood~\cite{wang2021moments}.
In the original Cohen-Lenstra-Martinet heuristics, fields are ordered by discriminant, which was an obvious ordering for number fields.
Now we have the question of what kind of ordering one should put on the fields.
In some cases, ordering fields by discriminant will contradict what is predicted by the heuristics.
See Bartel and Lenstra~\cite{bartel2020class} for example.
The counting function of number fields in this paper we mainly use is the product of ramified primes, which is usually denoted by $P(K)$ for a number field $K$.
See Wood~\cite{wood2010probabilities} for more discussions on different orderings.

Now let's give an example when the prime divides the order of the Galois group, which could be thought of as a prototype.
The genus theory for quadratic number fields implies the following:
\begin{equation}\label{eqn:genus theory for quadratic number fields}
    \omega(P(K))-1\leq\operatorname{rk}_2\operatorname{Cl}_K\leq\omega(P(K)),
\end{equation}
where $\omega(n)$ is the number of distinct prime divisors for $n\in\mathbb{Z}$.
We'll make a brief introduction to genus theory later.
See also Ishida~\cite{ishida1976genus}.
For now, using the inequality~(\ref{eqn:genus theory for quadratic number fields}) above, we can show that the distribution of the group $\operatorname{Cl}_K[2^\infty]$ cannot be described by Cohen-Lenstra Heuristics in the following sense.
First, roughly speaking, the inequality~(\ref{eqn:genus theory for quadratic number fields}) basically means that the number of ramified primes determines $\operatorname{rk}_2\operatorname{Cl}_K$ (up to $1$).
Second, technically, if we let $\mathcal{S}$ be the set of all quadratic number fields, then for each non-negative integer $r$, we have
\begin{equation*}
	\mathbb{P}_{\mathcal{S},P}(\operatorname{rk}_2\operatorname{Cl}_K\leq r)=0\quad\text{and}\quad\mathbb{E}_{\mathcal{S},P}(\lvert\operatorname{Hom}(\operatorname{Cl}_K,C_2)\rvert)=+\infty,
\end{equation*}
which is qualitatively different from what is predicted by the heuristics.
For example, the $C_3$-moment of quadratic number fields exists as a finite number.
See Davenport and Heilbronn~\cite{Davenport1971Cubic}.
According to this example, the statistical behaviour of $\operatorname{Cl}_K[2^\infty]$ should be reconsidered.
See Gerth's conjecture~\cite{Gerth1987} and Alex Smith's proof for the quadratic case~\cite{smith2022distribution} for more details.

Based on this example, an idea for the study of distribution of class groups is to apply the genus theory for general number fields and try to obtain similar results.
The introduction of genus theory is given in Section~\ref{section:non-randomness} following Ishida~\cite[Chapter 4]{ishida1976genus}.
Here we present the main result in a brief way.
\begin{definition}\label{def: e(p) and genus group}
	Let $K/\mathbb{Q}$ be a number field.
	\begin{enumerate}
		\item For each rational prime $p$, if we have ideal factorization $p\mathscr{O}_K={\mathfrak{p}}_1^{e_1}\cdots{\mathfrak{p}}_n^{e_n}$, then define
		\begin{equation*}
			e_K(p):=\gcd(e_1,\dots,e_n)
		\end{equation*}
		to be the greatest common divisor of its ramification indices.
		\item Define the \emph{genus group} $\mathcal{G}$ of $K$ to be the Galois group of the maximal unramified extension $Kk/K$ obtained by composing with an abelian number field $k/\mathbb{Q}$.
	\end{enumerate}	
\end{definition}
Then for a rational prime $q$, we have the following estimate on the $q$-rank of the the genus group.
\begin{theorem}\label{thm:S1 genus theory}
Let \(K/\mathbb{Q}\) be a number field with maximal abelian subextension $K_0/\mathbb{Q}$.
Fix a rational prime \(q\) dividing $n:=[K:\mathbb{Q}]$.
Then, the \(q\)-rank of the genus group \(\mathcal{G}\) admits the following inequality
\[\operatorname{rk}_q\mathcal{G}\geq\#\{p\mid e_K(p)\equiv0\bmod{q}\text{ and }p\equiv1\bmod{q}\}-\operatorname{rk}_q G(K_0/\mathbb{Q}).\]
\end{theorem}
Note that the genus group of $K$, by definition, is the quotient of the ideal class group.
Thereby we have an estimate for $\operatorname{rk}_q\operatorname{Cl}_K$.
On the other hand, the work of P.Roquette and H.Zassenhaus~\cite{RZ1969ClassRank} says that we can construct a subgroup of $\operatorname{Cl}_K$, which could be used for the estimate of the $q$-rank.
And the main theorem of their paper says the following.
\begin{theorem}\cite[Theorem 1]{RZ1969ClassRank}\label{thm:S1 invariant part}
Let \(K\) be a number field of degree \(n\) over \(\mathbb{Q}\) and \(q\) be a given prime number, then
\[\operatorname{rk}_q\operatorname{Cl}_K\geq\#\{p\mid e_K(p)\equiv0\bmod{q}\}-2(n-1)\]
\end{theorem}
For the introduction of their construction and more details, see Section~\ref{section:non-randomness}.

Since the splitting type of a prime $p$ in a field extension $K/\mathbb{Q}$ is determined by its decomposition group $G_p$ and inertia subgroup $I_p$ and so on, we can first study the group structure of the Galois group.
\begin{definition}\label{def: e(g)}
    Let $1\leq G\leq S_n$ be a finite transitive permutation group.
    Let $g\in G$ be any permutation.
    Define $e(g):=\gcd(\lvert\langle g\rangle\cdot1\rvert,\dots,\lvert\langle g\rangle\cdot n\rvert)$, i.e., the greatest common divisor of the size of orbits.
\end{definition}
\begin{example}
For now let $G=S_3$. 
If $K$ is a non-Galois cubic number field, then the permutation action of $G$ on $K\to\mathbb{C}$ is exactly the conventional action of $S_3$ on $\{1,2,3\}$, which induces the isomorphism $S_3\cong G(\hat{K}/\mathbb{Q})$, where $\hat{K}/\mathbb{Q}$ is the Galois closure of $K/\mathbb{Q}$.
By checking the elements of $S_3$, we see that $3$ is the only rational prime such that the set
\begin{equation*}
	\{g\in G\mid e(g)\equiv0\bmod{3}\}
\end{equation*}
is nontrivial.
On the other hand, a rational prime $p>3$ has inertia generator conjugate to $(123)$ if and only if it is totally ramified in $K/\mathbb{Q}$.
In other words, when $p>3$, we have that $e_K(p)\equiv0\bmod{3}$ if and only if $p$ is totally ramified.
So, Theorem~\ref{thm:S1 invariant part} says that
\begin{equation*}
	\operatorname{rk}_3\operatorname{Cl}_K\geq\#\{p\mid p\text{ is totally ramified}\}-4.
\end{equation*}
\end{example}
See Lemma~\ref{lemma: non-random prime} for more information.
This example indicates that the distribution of $\operatorname{rk}_3\operatorname{Cl}_K$ is closely related to the number of totally ramified primes for cubic fields $K/\mathbb{Q}$.
Inspired by this example, we make the following hypothesis on counting fields based on the Malle-Bhargava Heuristics.
See Malle~\cite{malle2004distribution}, Bhargava~\cite{bhargava07mass}, and Wood~\cite[p. 291-339]{directions2016} for the details of the heuristics.

From now on, we let $G$ be a transitive permutation group.
\begin{definition}\label{def: specifactions}
	\begin{enumerate}
		\item We say that a subset $\Omega$ of $G$ is closed under invertible powering if for each $g,h\in G$ there exists $a,b\in\mathbb{Z}$ such that $g^a=h$, $h^b=g$, then $g\in\Omega$ if and only if $h\in\Omega$.
		\item Let $\Omega\subseteq G$ be a nonempty subset that is closed under invertible powering and conjugation.
		Let $\mathcal{S}$ be a subset of $\mathcal{S}(G)$.
		Define for the set $\Omega$, and for all $r=0,1,2,\dots$, the subset $\mathcal{S}_{\Omega}^r$ of $\mathcal{S}$ as follows:
		A field $K\in\mathcal{S}$ is contained in $\mathcal{S}_{\Omega}^r$ if there exists exactly $r$ tamely ramified primes $p$ such that its inertia generator is conjugate to some $g\in\Omega$.
	\end{enumerate}	
\end{definition}
The hypothesis itself is a statement on field-counting.
\begin{hypothesis}\label{hypothesis main}
	Let $1\notin\Omega$ be a nonempty subset of $G$ that is closed under invertible powering.
	Let $\mathcal{S}:=\mathcal{S}(G)$ ordered by some counting function $C$.	
	If for each non-negative integer $r$, we have
	\begin{equation*}
		N_{\mathcal{S}_{\Omega}^r,C}(X)=o\left(N_{\mathcal{S},C}(X)\right)
	\end{equation*}
	then we say that the Hypothesis holds for the pair $({\mathcal{S}},C,\Omega)$.
\end{hypothesis}
Then let's present our result on the distribution of class groups based on this hypothesis.
In other words, either we assume the hypothesis or we should prove it.
Recall that for a number field $K$, if $L/K$ is an algebraic extension, then the norm map $\operatorname{Nm}_{L/K}:L\to K$ induces a map $\operatorname{Nm}_{L/K}:\operatorname{Cl}_L\to\operatorname{Cl}_K$.
Let $\operatorname{Cl}(L/K):=\ker(\operatorname{Nm}_{L/K}(\operatorname{Cl}_L\to\operatorname{Cl}_K))$ be the \emph{relative class group}.
For example, $\operatorname{Cl}(K/\mathbb{Q})=\operatorname{Cl}_K$ is just the usual class group.
\begin{theorem}\label{thm:statistical results for relative class groups}
	Let $k$ be a number field, let $\mathcal{S}:=\mathcal{S}(G,k)$ ordered by some counting function $C$.
	Let $H\subseteq G$ be a subgroup such that $\hat{K}^H\subseteq K$ for $K\in\mathcal{S}$, where $\hat{K}$ is the Galois closure of $K$.
	Let $q$ be a rational prime such that $q\mid[K:\hat{K}^H]$, and let
	\begin{equation*}
		\Omega_q:=\{g\in G\mid e(g)\equiv0\bmod{q}\}.
	\end{equation*}
	If Hypothesis~\ref{hypothesis main} holds for $(\mathcal{S},C,\Omega)$, where $\Omega$ is a nonempty subset of $\Omega_q$ that is closed under invertible powering and conjugation, then
	\begin{equation*}
		\mathbb{P}_{\mathcal{S},C}(\operatorname{rk}_q\operatorname{Cl}(K/\hat{K}^H)\leq r)=0\quad\text{and}\quad\mathbb{E}_{\mathcal{S},C}(\lvert\operatorname{Hom}(\operatorname{Cl}(K/\hat{K}^H),C_q)\rvert)=+\infty.
	\end{equation*}
\end{theorem}
For the proof of the above Theorem~\ref{thm:statistical results for relative class groups} and for more discussions on the distribution of class groups in general, see Section~\ref{section:non-random primes}.

With the help of Class Field Theory and Tauberian Theorems, we can prove the Hypothesis~\ref{hypothesis main} for abelian extensions.
To be precise, we have the following.
\begin{theorem}\label{thm:abelian case relative class groups}
	Assume that $G$ is abelian and let $H$ be a subgroup,
	and let ${\mathcal{S}}:=\mathcal{S}(G)$ ordered by the product of ramified primes $P$.
	If $q$ is a prime number and $l$ is a positive integer such that $q^l\mid\lvert G/H\rvert$, then the Hypothesis~\ref{hypothesis main} holds for $(\mathcal{S},P,\Omega_{q^l})$, where $\Omega_{q^l}:=\{g\in G\mid e(g)\equiv0\bmod{q^l}\}$.
	In addition, we have
	\begin{equation*}
		\mathbb{P}_{\mathcal{S},P}(\operatorname{rk}_q q^{l-1}\operatorname{Cl}(K/K^H)\leq r)=0\quad\text{and}\quad\mathbb{E}_{\mathcal{S},P}(\lvert\operatorname{Hom}(q^{l-1}\operatorname{Cl}(K/K^H),C_q)\rvert)=+\infty.
	\end{equation*}
\end{theorem}
See Section~\ref{section: abelian extensions} for the proof.
This result simply means that when the prime $q\mid\lvert G\rvert$, we have infinite $C_q$-moment for abelian $G$-fields.
Of course, this is different from what is predicted for $\operatorname{Cl}_K[p^\infty]$ when $p\nmid\lvert G\rvert$ by Cohen-Lenstra Heuristics.

For non-abelian extensions, the first obstacle is counting fields.
We present here an example, $D_4$-fields.
Let $D_4$ be the dihedral group of order $8$, and we are going to consider quartic number fields $L/\mathbb{Q}$ whose Galois closure $M/\mathbb{Q}$ are $D_4$-fields.
According to the work of S.A.Altug, A.Shankar, I.Varma, K.H.Wilson~\cite{altug2017number}, the result of counting such fields by \emph{the Artin conductor of the 2-dimensional irreducible representation of }$D_4$ is proven.
The main result in this case can be summarized as follows.
Note that $D_4$ viewed as a subgroup of $S_4$ is a transitive permutation group.
\begin{theorem}\label{thm: S1 D4 conductor}
	Let $\mathcal{S}:=\mathcal{S}(D_4)$.
	We have
	\begin{equation*}
		\mathbb{E}_{\mathcal{S},C}(\lvert\operatorname{Hom}(\operatorname{Cl}_L,C_2)\rvert)=+\infty,
	\end{equation*}
	where the subscript $C$ means that the fields $L\in\mathcal{S}$ are ordered by the Artin conductor of 2-dimensional irreducible representation of $D_4$.
\end{theorem}
The method we adopt follows from Altug and Shankar and Varma and Wilson~\cite{altug2017number}.
See Section~\ref{section:D4 extensions} for the proof.

To summarize this section, the notion ``good prime'' in Cohen and Martinet~\cite[D{\'e}finition 6.1]{CM90} (see also W. and Wood~\cite[Section 7]{wang2021moments}) gives a criterion for us to apply the Cohen-Martinet-Lenstra Heuristics so that we can predict the statistical behaviour of the class groups based on random modules.
However, if a rational prime $q$ satisfies the condition that
\begin{equation*}
	\{g\in G\mid e(g)\equiv0\bmod{q}\}\neq\emptyset,
\end{equation*}
then we expect a nontrivial subgroup of $\operatorname{Cl}_K$ associated to ramified primes and it admits qualitatively different statistical results from ``good prime'' cases.
For abelian fields and $D_4$-fields, this could be proved by field-counting.
For general dihedral extensions, we can prove it with some additional hypothesis.
See also Section~\ref{section:Dihedral extensions}.
In general, this is still remained unknown.
Finally, because of Gerth's conjecture~\cite{Gerth1987} and the work of Alex Smith~\cite{smith2022distribution}, we see that the study of the primes that are not good is also closely related to the Cohen-Lenstra-Martinet Heuristics itself.

\section{Basic notations}\label{section:notation}
In this section we introduce some of the notations that will be used in the paper.
We use some standard notations coming from analytic number theory.
For example, write a complex number as $s=\sigma+it$.
The notation $q^l\|n$ means that $q^l\mid n$ but $q^{l+1}\nmid n$.
Denote the Euler function by $\phi(n)$.
Let $\omega(n)$ counts the number of distinct prime factors of $n$ and so on.

We also follow the notations of inequalities with unspecified constants from Iwaniec and Kowalski~\cite[Introduction, p.7]{iwaniec2004analytic}.
Let's just write down the ones that are important for us.
Let $X$ be some set, and let $f,g$ be two (complex) functions.
Then $f(x)\ll g(x)$ for $x\in X$ means that $\lvert f(x)\rvert\leq Cg(x)$ for some constant $C\geq0$.
Any value of $C$ for which this holds is called an implied constant.
Note that we use Vinogradov's $f\ll g$ and Landau's $f=O(g)$ as synonyms.
We use $f(x)\asymp g(x)$ for $x\in X$ if $f(x)\ll g(x)$ and $g(x)\ll f(x)$ both hold with possibly different implied constants.
If $X$ admits a distant function, then we say that $f=o(g)$ as $x\to x_0$ if for any $\epsilon>0$ there exists some (unspecified) neighbourhood $U_\epsilon$ of $x_0$ such that $\lvert f(x)\rvert\leq\epsilon g(x)$ for $x\in U_\epsilon$.

\section{Non-randomness}\label{section:non-randomness}
The term non-randomness here refers to the case where we can obtain some nontrivial ideal classes in $\operatorname{Cl}_K$ associated to the ramified primes.
Let's first see what genus theory say in this case.

\subsection{Genus theory}
The goal of this section is to prove Theorem~\ref{thm:S1 genus theory}, which requires a brief introduction on genus theory for number fields.
The basic question of genus theory is to find out the maximal unramified abelian extension of a number field \(K\) obtained by composing with an absolute abelian number field \(k\).
To be precise, we have the following definition.
\begin{definition}
let \(K/\mathbb{Q}\) be a number field of degree \(n\),
and let $k/\mathbb{Q}$ be the maximal abelian extension such that $Kk/K$ is an unramified extension.
We call the compositum $Kk$ the \emph{genus field} over \(K\),
and call the Galois group \(G(Kk/K)\) the \emph{genus group} \(\mathcal{G}\).
\end{definition}
Since \(Kk/K\) is an unramified abelian extension, it is a subfield of the Hilbert class field of \(K\) whose Galois group is isomorphic to the class group \(\operatorname{Cl}_K\) of \(K\), 
hence the genus group is a quotient group of the class group \(\operatorname{Cl}_K\).
According to Ishida~\cite[p.33-39]{ishida1976genus}, we make a summary of the results on the genus group.
Fix a rational prime \(q\) such that \(q^t\| n\) with \(t\geq1\).
Let $A\cong G(k/\mathbb{Q})$ be the Galois group of the abelian extension $k/\mathbb{Q}$,
and let $K_0$ be the intersection of $k$ and $K$ which is also the maximal abelian subextension of $K/\mathbb{Q}$.
See the diagram for summary below.
\[\begin{tikzcd}
&Kk\arrow[dr,dash]\arrow[dl,dash,"\mathcal{G}"]&\\
K\arrow[dr,dash]&&k\arrow[ddl,dash,"A"]\\
&K_0\arrow[d,dash]\arrow[ru,dash,"\mathcal{G}"]&\\
&\mathbb{Q}&
\end{tikzcd}\]
Let's state the result as follows.
\begin{theorem}\label{thm:q-rank estimation of genus group}
Let \(K/\mathbb{Q}\) be a number field with maximal abelian subextension $K_0/\mathbb{Q}$.
Let $q$ be a fixed rational prime dividing $n:=[K:\mathbb{Q}]$ and let $l\geq1$ be some integer, the \(q\)-rank of the group \(q^{l-1}\mathcal{G}\) admits the following inequality
\begin{equation*}
	\operatorname{rk}_q q^{l-1}\mathcal{G}\geq\#\{p\mid e_K(p)\equiv0\bmod{q^l}\text{ and }p\equiv1\bmod{q}\}-\operatorname{rk}_q G(K_0/\mathbb{Q}).
\end{equation*}
\end{theorem}
\begin{remark}
Theorem~\ref{thm:S1 genus theory} just follows from this result.
Moreover we can see that if $\mathcal{G}[q^\infty]$ admits higher torsion part, then $\operatorname{Cl}_K[q^\infty]$ must also have higher torsion part.
\end{remark}
\begin{proof}
	This result follows from Ishida~\cite[Chapter IV, Theorem 3]{ishida1976genus}.
\end{proof}
\begin{example}
Let $K/\mathbb{Q}$ be a non-Galois cubic field, and $q$ equal $3$.
Then the requirement $\gcd(p-1,e_K(p))\equiv0\bmod{3}$ is equivalent to $p$ being totally ramified in $K/\mathbb{Q}$ and $p\equiv1\bmod{3}$.
In other words, we have
\begin{equation*}
\operatorname{rk}_3\operatorname{Cl}_K\geq\#\{p\text{ is a totally ramified prime and }p\equiv1\bmod{3}\}.
\end{equation*}
This clearly generalizes genus theory for the quadratic case.
See also \cite[Chapter 5]{ishida1976genus} for more discussions on the case of odd prime degree.
\end{example}

\subsection{The class rank estimate on the invariant part of the class group}
In the paper of Roquette and Zassenhaus~\cite{RZ1969ClassRank}, there is another result on the estimate of the \(q\)-rank of the class group with respect to ramified primes.
As seen in previous section, what genus theory is focused on is the genus group, which is the quotient of the class group, while the invariant part of the class group we will talk about is a subgroup.
	
We will explain briefly the work of Roquette and Zassenhaus by showing the proof of Theorem~\ref{thm:rz69}.
More importantly, we want to show the construction, which is more useful in this paper.
Let's first introduce some notations.
If $K$ is a number field, let $I_K$ be the group of fractional ideals and $P_K$ be the group of principal ideals, and let $\operatorname{rk}_K$ be its rank of global units.
Let \(v_q(n)\) denote the normalized exponential valuation associated to the rational prime $q$, i.e., we compute the exponent of $q$ showing up in the factorization of $n$.
\begin{definition}\label{def:invariant part of class groups}
Let $K/\mathbb{Q}$ be a number field whose Galois closure is a $G$-extension $L/\mathbb{Q}$.
By viewing the group of fractional ideals $I_K$ of $K$ as a subgroup of $I_L$, 
define
\begin{equation*}
	D_K^G:=\operatorname{im}(I^G_L\cap I_K\to\operatorname{Cl}_K).
\end{equation*}
Let's call it the \emph{invariant part of the class group}.
\end{definition}
Let's give the statement on the structure of the subgroup $D^G_K[q^\infty]\subseteq\operatorname{Cl}_K[q^\infty]$, which is due to Roquette and Zassenhaus.
\begin{theorem}\label{thm:rz69}
Let $K/\mathbb{Q}$ be a (not necessarily Galois) extension of degree $n$ such that its Galois closure is a $G$-extension.
Let \(q\) be a fixed prime number, and let $l\geq1$ be an integer, then
\[\begin{aligned}
&\#\{p\mid e_K(p)\equiv0\bmod{q^l}\}-2(n-1)\\
&\leq\operatorname{rk}_qq^{l-1}D_K^G\leq\\
&\#\{p\mid e_K(p)\equiv0\bmod{q^l}\}.
\end{aligned}\]
\end{theorem}
We prove the theorem by two lemmas.
The first lemma gives a detailed description of the elements in $D_K^G$ in the sense of fractional ideals.
For each rational prime $p$, if $p\mathscr{O}_K=\mathfrak{p}_1^{e_1}\cdots\mathfrak{p}_r^{e_r}$ is the factorization in $K$, then define
\begin{equation*}
	\mathfrak{a}(p):=\prod_{i=1}^r\mathfrak{p}_i^{e_i/e_K(p)}.
\end{equation*}
\begin{lemma}\label{lemma:structure of invariant part}\cite[Equation (8)]{RZ1969ClassRank}
Let $K/\mathbb{Q}$ be a number field whose Galois closure is a $G$-extension $L/\mathbb{Q}$.
\begin{enumerate}
	\item The group $I_K\cap I_L^{G}$ is free abelian generated by 
	\begin{equation*}
		\{\mathfrak{a}(p)\mid p\text{ is a rational prime}\}.
	\end{equation*}
	\item The group $\tilde{D}_K^G:=I_K\cap I_L^{G}/P_{\mathbb{Q}}$ is isomorphic to $\prod_p\mathbb{Z}/e_K(p)$, where $P_k$ means the group of principal ideals of a number field $k$.
\end{enumerate}
\end{lemma}
The second lemma is to estimate the difference between principal ideals of $K$ and $\mathbb{Q}$.
\begin{lemma}\label{lemma:difference between principal ideals}\cite[Equation(11)]{RZ1969ClassRank}
Let $K/\mathbb{Q}$ be a number field whose Galois closure a $G$-extension $L/\mathbb{Q}$, and let $\mathscr{P}_k$ be the group of principal ideals of a number field $k$.
Then
\begin{equation}\label{eqn:invariant}
    \operatorname{rk}_q{P}_K^G/P_{\mathbb{Q}}\leq2(n-1),
\end{equation}
where $P_K^G:=P_K\cap I_L^G$.
\end{lemma}
Now let's prove the theorem.
\begin{proof}[Proof of Theorem~\ref{thm:rz69}]
Lemma~\ref{lemma:structure of invariant part} gives the upper bound of $\operatorname{rk}_qq^{l-1}D_K^G$, because $P_{\mathbb{Q}}\subseteq P_K$.
According to the short exact sequence
\begin{equation}
	0\to P_K^G/P_{\mathbb{Q}}\to I_K^G/P_{\mathbb{Q}}\to D_K^G\to0,
\end{equation}
the inequality (\ref{eqn:invariant}) tells us the lower bound directly.
\end{proof}
We can apply the theorem to the class group $\operatorname{Cl}_K$, for $D_K^G\subseteq\operatorname{Cl}_K.$
\begin{corollary}
Let \(K\) be a number field of degree \(n\) over \(\mathbb{Q}\), and let \(q\) be a prime number, and let $l\geq1$ be an integer, then
\[\operatorname{rk}_qq^{l-1}\operatorname{Cl}_K\geq\#\{p\mid e_K(p)\equiv0\bmod{q^l}\}-2(n-1).\]
\end{corollary}
One of the advantages of this theory, as mentioned above, is that $D_K^G$ is a subgroup of $\operatorname{Cl}_K$.
So, we can even try to discuss the relative class group here, i.e., for a subfield $K'\subseteq K$, we want to give a description for $D_K^G\cap\operatorname{Cl}(K/K')$.
\begin{theorem}\label{thm:rz69 rel version}
Let $K/\mathbb{Q}$ be a number field of degree $n$ whose Galois closure is a $G$-extension $L/\mathbb{Q}$.
If $K'\subseteq K$ such that $q^l\mid[K:K']$ where $q$ is a rational prime, and $l\geq1$, then
\[\operatorname{rk}_qq^{l-1}\operatorname{Cl}(K/K')\geq\#\{p\mid e_K(p)\equiv0\bmod{p^l}\}-2(n-1).\]
\end{theorem}
\begin{proof}
For each prime $p$, recall that $\mathfrak{a}(p)$ is the ideal of $K$ such that $p\mathscr{O}_K=\mathfrak{a}(p)^{e_K(p)}$.
We've shown in Lemma~\ref{lemma:structure of invariant part} that $\mathfrak{a}(p)$ is fixed by the action of $G$, viewed as an element of $I_L$.
So, we have the following computation
\[\operatorname{Nm}_{K/K'}(\mathfrak{a}(p))=\mathfrak{a}(p)^{[K:K']},\]
where $\mathfrak{a}(p)^{[K:K']}$ is treated as an ideal of $K'$.
If $q^{l}\mid e_K(p)$, and $\mathfrak{b}(p):=(\mathfrak{a}(p))^{e_K(p)/q^{l}}$, then $\operatorname{Nm}_{K/K'}(\mathfrak{b}(p))$ becomes a power of $p\mathscr{O}_{K'}$, hence a principal ideal of $K'$.
So, $\mathfrak{b}(p)$ represents an ideal class of $\operatorname{Cl}(K/K')$.
Let's define
\begin{equation*}
	I_{L/K}^G[q^l]:=\langle\mathfrak{b}(p)\mid e_{K}(p)\equiv0\bmod{q^l}\rangle.
\end{equation*}
Note that $I_{L/K}^G[q^l]\subseteq I_K^G$, and we have
\begin{equation*}
	I_{L/K}^G[q^l]P_{\mathbb{Q}}/P_{\mathbb{Q}}\cong\prod_{i=1}^s\mathbb{Z}/q^{l}
\end{equation*}
where $s=\#\{p\mid e_{K}(p)\equiv0\bmod{q^l}\}$.
By Lemma~\ref{lemma:structure of invariant part}, the group $D_K^G$ is generated by ideals of the form $\mathfrak{a}(p)$, we therefore know that
\begin{equation*}
	D_{L/K}^G[q^l]:=I_{L/K}^G[q^l]P_K/P_K\subseteq D_K^G[q^l]\cap\operatorname{Cl}(K/K').
\end{equation*}
By Lemma~\ref{lemma:difference between principal ideals}, we know that $\operatorname{rk}_q P^G_K/P_{\mathbb{Q}}$ is bounded above by a constant $2(n-1)$.
Using the short exact sequence
\begin{equation*}
	1\to
	P_K^G/P_{\mathbb{Q}}\cap I_{L/K}^G[q^l]P_{\mathbb{Q}}/P_{\mathbb{Q}}
	\to I_{L/K}^G[q^l]P_{\mathbb{Q}}/P_{\mathbb{Q}}
	\to D_{L/K}^G[q^l]
	\to1,
\end{equation*}
we obtain the desired result.
\end{proof}

\section{Non-random primes}\label{section:non-random primes}
Let $q$ be a fixed rational prime.
In Theorem~\ref{thm:S1 genus theory} and Theorem~\ref{thm:S1 invariant part}, we see that $\operatorname{rk}_q\operatorname{Cl}_K$ is related to the set
\begin{equation*}
	\{p\mid e_K(p)\equiv0\bmod{q}\}.
\end{equation*}
The ramification of a prime $p$ in the extension $K/\mathbb{Q}$ could be recovered from the decomposition group $G_p=G_p(K/\mathbb{Q})$ and the inertia subgroup $I_p$ and so on.
So the condition $e_K(p)\equiv0\bmod{q}$ could be translated as a condition on the Galois group $G$.

Let $G$ be a transitive permutation group in this section.
In a $G$-extension (see Definition~\ref{def:Gamma-fields}), let's first give the translation by the following lemma.
\begin{lemma}\label{lemma: non-random prime}
    Let $q$ be a fixed prime,
    and let $(K/\mathbb{Q},\psi)$ be an extension of number fields such that its Galois closure $(\hat{K},\psi)$ is a $G$-field and that $K=\hat{K}^{G_1}$, where $G_1$ is the image of stabilizer of $1$.
    For a fixed positive integer $l$, a rational prime $p\nmid\lvert G\rvert$ satisfies the identity
    \begin{equation*}
    	e_K(p)\equiv0\bmod{q^l},
    \end{equation*} 
    if and only if $I_p\cap\Omega_{q^l}\neq\emptyset$ where $I_p$ is the inertia subgroup of $p$ (up to conjugate), and $\Omega_{q^l}:=\{g\in G\mid e(g)\equiv0\bmod{q^l}\}$.
\end{lemma}
\begin{proof}
	Let $\mathfrak{P}$ be a prime of $\hat{K}$ above $p$.
    One can check that if $G_{\mathfrak{P}}$ is the decomposition group of $\mathfrak{P}$, then the orbits of $G_{\mathfrak{P}}$ on $G/G_1$ correspond to the primes $\mathfrak{p}$ of $K$ above $p$, and the size of orbits of $I_{\mathfrak{P}}$ corresponds to the ramification index respectively.
    
    If $p$ is a tamely ramified prime, then $I_{\mathfrak{P}}=\langle g\rangle$ is just cyclic.
    So, $e(p)\equiv0\bmod{q^l}$ if and only if the size of orbits of the inertia generator has common divisor $q^l$, i.e., $q^l\mid e(g)$ using the notations in Definition~\ref{def: e(p) and genus group} and \ref{def: e(g)}, which means that $I_{\mathfrak{P}}\cap\Omega_{q^l}\neq\emptyset$.
    Note that $I_p$ is well-defined up to conjugation.
    And the result just follows.
\end{proof}
So roughly speaking, if $q$ is a prime such that the set
\begin{equation*}
	\Omega_q=\{g\in G\mid e(g)\equiv0\bmod{q}\}
\end{equation*}
is nontrivial, then potentially results like Theorem~\ref{thm:S1 invariant part} could give us some nontrivial estimate for $\operatorname{rk}_q\operatorname{Cl}_K$, when $K$ runs over fields in $\mathcal{S}(G)$.
The discussion of this section is intended for such primes, and philosophically they should lead us to results like infinite $C_q$-moment introduced in Section~\ref{section:intro}.
And this explains the title of the section.

Our most important goal in this section is to prove Theorem~\ref{thm:statistical results for relative class groups}.
We reach the theorem in several steps.
The first step is to show that the Hypothesis~\ref{hypothesis main} implies zero-probability of $\operatorname{rk}_q\operatorname{Cl}_K\leq r$.
\begin{theorem}\label{thm:P(rk X=r)=0}
Let $\mathcal{S}:=\mathcal{S}(G)$ ordered by a counting function $C$.
Let $G_1\subseteq H\subseteq G$ be a tower of subgroups of $G$, where $G_1$ is the stabilizer of $1$.
Let $q$ be a rational prime such that $q^l\mid[H:G_1]$, where $l\geq1$.
Let $\Omega$ be a subset of $G$ that is closed under invertible powering and conjugation such that for each $g\in\Omega$ we have that $e(g)\equiv\bmod{q^l}$.
If the Hypothesis~\ref{hypothesis main} hold for $(\mathcal{S},C,\Omega)$, i.e., for all non-negative integer $r$, we have
\begin{equation*}
	N_{\mathcal{S}_{\Omega}^r,C}(X)=o\left(N_{\mathcal{S},C}(X)\right),
\end{equation*}
then for all $r$, we have
\begin{equation*}
	\mathbb{P}_{\mathcal{S},C}(\operatorname{rk}_qq^{l-1}\operatorname{Cl}(K/\hat{K}^{H})\leq r)=0.
\end{equation*}
\end{theorem}
\begin{proof}
	Write $n:=[K:\mathbb{Q}]$.
	According to Theorem~\ref{thm:rz69 rel version} and Lemma~\ref{lemma: non-random prime}, for a fixed extension $K/\mathbb{Q}$ in $\mathcal{S}$, we have that
	\begin{equation*}
	    [\#\{p\nmid\lvert G\rvert:I_p\cap\Omega\neq\emptyset\}\geq r+2n]
	    \Rightarrow
	    [\operatorname{rk}_qq^{l-1}\operatorname{Cl}(K/\hat{K}^{H})>r].
    \end{equation*}
    The condition says that for all $r$ we have $N_{\mathcal{S}^r_{\Omega},C}(X)=o\left(N_{\mathcal{S},C}(X)\right)$.
    This implies that
    \begin{equation*}
    	\begin{aligned}
    		&\mathbb{P}_{\mathcal{S},C}(\operatorname{rk}_qq^{l-1}\operatorname{Cl}(K/\hat{K}^{H})\leq r)
    		\\
    		=&\lim_{X\to\infty}
    		\frac{
    			\#\{K\in\mathcal{S}\mid C(K)<X\text{ and }\operatorname{rk}_q q^{l-1}\operatorname{Cl}(K/\hat{K}^H)\leq r\}
    		}{
    			N_{\mathcal{S},C}(X)
    		}
    		\\
    		\leq&\lim_{X\to\infty}\frac{\sum_{i=0}^{r+2n}N_{\mathcal{S}^i_{\Omega},C}(X)}{N_{\mathcal{S},C}(X)}=0.
    	\end{aligned}
    \end{equation*}
\end{proof}
Then let's prove that if the probability of $\operatorname{rk}_q\operatorname{Cl}_K\leq r$ is bounded, then we have infinite moment.
\begin{theorem}\label{thm:infinite moment for bad primes}
Let $\mathcal{S}:=\mathcal{S}(G)$ with a counting function $C$.
Let $G_1\subseteq H\subseteq G$ be a tower of subgroups of $G$, where $G_1$ is the image of stabilizer of $1$.
Let $q$ be a rational prime and $l\geq1$ be an integer.
If there exists some constant $0\leq c<1$ such that for all $r=0,1,2,\dots$, we have
\[\mathbb{P}_{\mathcal{S},C}(\operatorname{rk}_qq^{l-1}\operatorname{Cl}(K/\hat{K}^{H})\leq r)\leq c,\]
then
\begin{equation*}
	\mathbb{E}_{\mathcal{S},C}(\lvert\operatorname{Hom}(q^{l-1}\operatorname{Cl}(K/\hat{K}^{H}),C_q)\rvert)=+\infty,
\end{equation*}
where $\hat{K}$ is the Galois closure of $K$, and $C_q$ is the cyclic group of order $q$.
\end{theorem}
\begin{proof}
	Write
	\begin{equation*}
		H_{\mathcal{S},C}(X):=
		\sum_{
			\substack{
				K\in\mathcal{S}\\
				C(K)<X
			}
		}\lvert\operatorname{Hom}(q^{l-1}\operatorname{Cl}(K/\hat{K}^{H}),C_q).
	\end{equation*}
	For each non-negative integer $r$, define
	\begin{equation*}
		R^{r}_{\mathcal{S},C}(X):=
		\#\{K\in\mathcal{S}\mid C(K)<X\text{ and }\operatorname{rk}_q q^{l-1}\operatorname{Cl}(K/\hat{K}^{H})=r\}.
	\end{equation*}
	Let $d$ be any fixed positive integer.
	By definition,
	\begin{equation*}
		\lim_{X\to\infty}\mathbb{P}_{\mathcal{S},C}(\operatorname{rk}_q q^{l-1}\operatorname{Cl}(K/\hat{K}^{H})\leq d)
		=\lim_{X\to\infty}\sum_{i=0}^{d}\frac{R^{i}_{\mathcal{S},C}(X)}{N_{\mathcal{S},C}(X)}.
	\end{equation*} 
	So, for each small number $\epsilon>0$, there exists some $N>0$ such that for all $X>N$, we have that
	\begin{equation*}
		\frac{\sum_{i=0}^{d}R^{l,i}_{\mathcal{S},C}(X)}{N_{\mathcal{S},C}(X)}\leq c+\epsilon.
	\end{equation*}
	Note that for a fixed $X$, the sum
	\begin{equation*}
		\sum_{i=0}^\infty\frac{R^{l,i}_{\mathcal{S},C}(X)}{N_{\mathcal{S},C}(X)}
	\end{equation*}
	is finite and equal to $1$.
	So, for each $X>N$, we have
	\begin{equation*}
		\begin{aligned}
			\frac{H_{\mathcal{S},C}(X)}{N_{\mathcal{S},C}(X)}
			=&\sum_{i=0}^\infty q^i\frac{R^{i}_{\mathcal{S},C}(X)}{N_{\mathcal{S},C}(X)}
			\\
			\geq&q^d(1-\sum_{i=0}^d\frac{R^{i}_{\mathcal{S},C}(X)}{N_{\mathcal{S},C}(X)})
			\\
			\geq& q^d(1-c-\epsilon).			
		\end{aligned}
	\end{equation*}
	By letting $d$ go to infinity, we see that for each positive number $M>0$, there exists some constant $N>0$ such that for all $X>N$, the ratio
	\begin{equation*}
		\frac{H_{\mathcal{S},C}(X)}{N_{\mathcal{S},C}(X)}>M,
	\end{equation*}
	hence the result.
\end{proof}
Before we move on to more concrete situations, let's discuss an example where the fields are not ordered by product of ramified primes.
\begin{example}[Discriminant]
Let's consider the set $\mathcal{S}:=\mathcal{S}(S_3)$ of non-Galois cubic extensions.
Let $\Omega:=\{(123),(132)\}$.
In this case we want to show that the Hypothesis~\ref{hypothesis main} is \emph{false} for $(\mathcal{S},d,\Omega)$ when the fields ordered by the \emph{absolute discriminant} $d$, despite the fact that $(123)$ clearly satisfies the condition that $3\mid e((123))$.

According to the work of Bhargava, Shankar, and Tsimerman~\cite[Theorem 8]{bhargava2013davenport}, there exists some positive constants $c_1,c_2$ such that
\[N_{\mathcal{S}_{\Omega}^0,d}(X)\sim c_1X\quad\text{and}\quad N_{\mathcal{S},d}(X)\sim c_2 X,\]
when the cubic fields are ordered by discriminant.
This already contradicts the Hypothesis~\ref{hypothesis main} when ordering fields by discriminant.
On the other hand, Proposition~\ref{prop:Dq-case} shows that, under some other assumptions deduced from Cohen-Lenstra Heuristics and Malle-Bhargava Heuristics, the Hypothesis~\ref{hypothesis main} holds for $(\mathcal{S},P,\Omega)$, where $P$ is the product of ramified primes.
This shows that the hypothesis itself is dependent on the ordering of number fields,
and we have to choose a suitable counting function.
\end{example}

\section{Dirichlet series and Tauberian Theorem}\label{section:dirichlet series and Tauberian Theorem}
In this section, we discuss the analytic properties of some Dirichlet series, which will be useful tools in later sections. 
\subsection{Tauberian Theorem}
Let's first present a Tauberian Theorem that is used in the paper repeatedly.
Recall that for the complex variable $s\in\mathbb{C}$, we write $s=\sigma+it$.
\begin{theorem}[Delange-Ikehara]\label{thm:Delange-Ikehara}\cite[Appendix II Theorem I]{narkiewicz2014elementary}
    Assume that the coefficients of a Dirichlet series are real and non-negative, and that it converges in the half-plane $\sigma>1$, defining a holomorphic function $f(s)$.
    Assume, moreover, that in the same half-plane one can write
    \[f(s)=\sum_{j=0}^qg_j(s)\log^{b_j}\left(\frac{1}{s-1}\right)(s-1)^{-\alpha_j}+g(s),\]
    where functions $g,g_0,\dots,g_q$ are holomorphic in the closed half plane $\sigma\geq1$, the $b_j$-s are non-negative integers, $\alpha_0$ is a positive real number, $\alpha_1,\dots,\alpha_q$ are complex numbers with $\Re\alpha_j<\alpha_0$, and $g_0(1)\neq0$.
    
    Then for the summatory function $A(x)=\sum_{n<x}a_n$ we have, for $x$ tending infinity,
    \[A(x)=\left(\frac{g_0(1)}{\Gamma(\alpha_0)}+o(1)\right)x(\log x)^{\alpha_0-1}(\log\log x)^{b_0}.\]
    If $f$ satisfies the same assumptions, except that $\alpha_0=0$ and $b_0\geq1$, then
    \[A(x)=(b_0g_0(1)+o(1))x\frac{(\log\log x)^{b_0-1}}{\log x}.\]
\end{theorem}
See the paper of Delange~\cite{Delange54} for more results on Tauberian Theorems.
\subsection{Dirichlet L-series and some applications}
Recall that if $m>1$ is an integer and $\chi:(\mathbb{Z}/m\mathbb{Z})^*\to S^1$ is a Dirichlet character mod $m$, then the \emph{Dirichlet L-series} associated to $\chi$ is the series
\begin{equation*}
	L(\chi,s):=\sum_{n=1}^{\infty}\frac{\chi(n)}{n^s}.
\end{equation*}
It admits many analytic properties similar to those of Riemann zeta function \(\zeta(s)\).
Let's summarize them in the following statement.
\begin{proposition}
	\begin{enumerate}
		\item The series \(L(\chi,s)\) converges absolutely and uniformly in the domain \(\sigma\geq1+\delta\), for any \(\delta>0\).
		It therefore represents a holomorphic function on the half-plane \(\sigma>1\).
		\item When $\sigma>1$, we have \emph{Euler's identity}
		\begin{equation*}
			L(\chi,s)=\prod_{p}\frac{1}{1-\chi(p)p^{-s}}.
		\end{equation*}
		\item If $\chi$ is nontrivial, then it could be extended analytically to the whole complex plane.
		\item For any Dirichlet character $\chi$, the series $L(\chi,s)$ admits a zero-free region containing the closed half-plane $\sigma\geq1$.
	\end{enumerate}	
\end{proposition}
One can check its basic properties in some textbooks, say Neukirch~\cite[VII.2]{neukirch2013algebraic}.
The zero-free region could be viewed as a corollary of the zero-free region of Hecke $L$-series.
See Coleman~\cite{Coleman_1990} for example.
Let's prove an analytic tool that will be used later.
Recall that $\phi$ is the Euler's phi function.
\begin{lemma}\label{lemma: zeta((q,a);s)} 
	Let $m>1$ be an integer.
	Let $f:\mathbb{Z}/m\mathbb{Z}\to\mathbb{N}$ be a map such that $f(\alpha)\neq0$ for some $\alpha\in(\mathbb{Z}/m\mathbb{Z})^*$.
	By composing $\mathbb{Z}\to\mathbb{Z}/m\mathbb{Z}$, we also treat $f$ as a map $\mathbb{N}\to\mathbb{N}$.
	Define a Euler product
	\begin{equation*}
		F(s):=\prod_{p\text{ rational prime}}(1+f(p)p^{-s}).
	\end{equation*}
	Then $F(s)$ is absolutely and uniformly convergent in $\sigma>1+\delta$, for each $\delta>0$, and it represents a holomorphic function in the open half-plane $\sigma>1$.
	If $f$ is nonzero on $(\mathbb{Z}/m\mathbb{Z})^*$, then $F(s)$ admits an analytic continuation to the line $\sigma=1$ except at $s=1$.
	In particular, there exists some holomorphic function $g(s)$ at $s=1$ with $g(1)\neq0$ such that
	\begin{equation*}
		F(s)=g(s)(s-1)^{-\beta}
	\end{equation*}
	near $s=1$, where $\beta$ is given by the formula
	\begin{equation*}
		\beta=\frac{1}{\phi(m)}\sum_{\alpha\in(\mathbb{Z}/m)^*}f(\alpha).
	\end{equation*}
\end{lemma}
\begin{proof}
	Clearly $f$ is bounded above by some constant $N>0$, hence 
	\begin{equation*}
		\lvert F(s)\rvert\leq\prod_p(1+Np^{-\sigma}),
	\end{equation*}
	which implies that $F(s)$ is absolutely and uniformly convergent in $\sigma>1+\delta$, for each $\delta>0$.
	So it represents a holomorphic function in the half-plane $\sigma>1$.
	
	Let's write
	\begin{equation*}
		F(s)=\prod_{\alpha\in\mathbb{Z}/m\mathbb{Z}}\prod_{p\equiv\alpha\bmod{m}}(1+f(\alpha)p^{-s})=:\prod_{\alpha}F_\alpha(s).
	\end{equation*}
	It is clear that there are only finitely many primes $p\mid m$.
	And for a fixed factor $(1+f(p)p^{-s})$, it is a nonzero holomorphic function in the closed half-plane $\sigma\geq1$.
	So it suffices to prove the statement for each $F_\alpha(s)$ with $\alpha\in(\mathbb{Z}/m\mathbb{Z})^*$.
	
	Let's assume without loss of generality that $F(s)=F_\alpha(s)$ for some $\alpha\in(\mathbb{Z}/m\mathbb{Z})^*$.
	Since $F(s)$ is nonzero in the open half-plane $\sigma>1$, we can take the logarithm.
	For the primes $p>N$, we have that
	\begin{equation*}
		\lvert\log\prod_{p>N}(1+f(\alpha)p^{-s})-\frac{f(\alpha)}{\phi(m)}\sum_{\chi}\chi(\alpha)\log(\bar{\chi},s)\rvert\leq\sum_{p\leq N}\frac{f(\alpha)}{p^{\sigma}}+\sum_{p>N}\frac{2f(\alpha)}{p^{2\sigma-\epsilon}},
	\end{equation*}
	where $\chi$ runs over all characters modulo $m$, and $\epsilon>0$ is some small real number.
	The zero-free region of the Dirichlet $L$-series first shows that if $\chi$ is nontrivial then $\log L(\chi,s)$ is a holomorphic function in the closed half-plane $\sigma\geq1$.
	And the case when $\chi=\chi_0$ is the trivial character reduces to the Riemann zeta function.
	To be precise, $L(\chi_0,s)$ and $\zeta(s)$ differs by finitely many Euler factors, so it is a holomorphic function in $\mathbb{C}\backslash\{0,1\}$ with a simple pole at $s=1$.
	In particular, it is nonzero in the closed half-plane $\sigma\geq1$.
	So $F(s)$ could be extended analytically to the line $\sigma=1$ except at $s=1$.
	And the function defined by the formula
	\begin{equation*}
		g(s):=(s-1)^{\beta}F(s)
	\end{equation*}
	is a nonzero holomorphic function at $s=1$, hence the result.
\end{proof}
From the above proof, we obtain the following.
\begin{corollary}\label{cor:count (q,a) primes}
    Let $m>1$ be an integer.
    For each $\alpha\in(\mathbb{Z}/m\mathbb{Z})^*$, define
    \begin{equation*}
    	l_{m}^{1}(\alpha,s):=\sum_{p\equiv\alpha\bmod{m}}p^{-s}.
    \end{equation*}
    Then $l_{m}^1(s)$ is absolutely and uniformly convergent in the open half-plane $\sigma>1+\delta$ for any $\delta>0$, and it represents a holomorphic function in the half-plane $\sigma>1$.
    It could be extended analytically to the line $\sigma=1$ except at $s=1$.
    In particular, there exists some holomorphic function $g(x)$ at $s=1$ such that
    \begin{equation*}
    	l_{m}^1(\alpha,s)=\frac{1}{\phi(m)}\log\frac{1}{s-1}+g(s)
    \end{equation*}
    near $s=1$.
\end{corollary}
\begin{proof}
	Just compare $l_m^1(\alpha,1)$ and
	\begin{equation*}
		\frac{1}{\phi(m)}\sum_{\chi}\log L(\chi,s)
	\end{equation*}
	where $\chi$ runs over all primitive characters mod $m$.
\end{proof}
As a generalization, we have the following.
\begin{proposition}\label{prop: count (q,a) primes}
	Let $m>1$ be an integer.
	Let $f:\mathbb{N}\to\mathbb{N}$ be a multiplicative map such that for all $p\equiv q\bmod{m}$ we have that $f(p)=f(q)$ and we treat it also as a map defined on $(\mathbb{Z}/m\mathbb{Z})$.
	Define the series
	\begin{equation*}
		l_{f}^r(s):=
		\sum_{
			\substack{
				n\text{ square-free}\\
				\omega(n)=r
			}
		}f(n)n^{-s}.
	\end{equation*}
	Then $l_{f}^r(s)$ defines a holomorphic function in the open half-plane $\sigma>1$.
	If $f$ is nonzero on $(\mathbb{Z}/m\mathbb{Z})^*$, then $l_{f}^r(s)$ could be extended analytically to the line $\sigma=1$ except at $s=1$.
	In particular, there exists holomorphic functions $g_0(s),\dots,g_{r-1}(s)$ and a nonzero constant function $g_r(s)$ at $s=1$ such that
	\begin{equation*}
		l_{f}^r(s)=\sum_{j=0}^{r}g_j(s)\Bigl(\log\frac{1}{s-1}\Bigr)^j.
	\end{equation*}
\end{proposition}
\begin{proof}
	First of all, if $f$ is zero on $(\mathbb{Z}/m\mathbb{Z})^*$, then clearly $l_f^r(s)$ is a holomorphic function in the half-plane $\sigma>0$, for there are only finitely many primes $p\mid m$.
	Let's start with the simplest case where $f(\alpha_0)=1$ for a fixed $\alpha_0\in(\mathbb{Z}/m\mathbb{Z})^*$ and $f(\alpha)=0$ for all $\alpha\neq\alpha_0$ in $\mathbb{Z}/m\mathbb{Z}$.
	In this case, let's just write
	\begin{equation*}
		l_m^r(\alpha_0,s):=l^r_f(s).
	\end{equation*}
	If $r=1$, then this is exactly the Corollary~\ref{cor:count (q,a) primes}.
	Let's prove the statement by induction on $r$.
	Assume that $l_m^j(\alpha_0,s)$ satisfies the properties in the statement, for each $j=1,2,\dots,r-1$.
	One can check that the following identity holds:
	\begin{equation*}
		l_m^r(\alpha_0,s)=\frac{1}{r}l_m^1(\alpha_0,s)l_{m}^{r-1}(\alpha_0,s)+\sum_{j=2}^r\frac{(-1)^{j}}{r}l_{m}^1(\alpha_0,js)l_{m}^{r-j}(\alpha_0,s).
	\end{equation*}
	Then the property of $l_m^r(\alpha_0,s)$ just follows from the induction assumption and the above identity.
	
	Now assume that $f$ is any function that satisfies the condition in the statement.
	Let $(r_{\alpha})_{\alpha}$ denote a point in $\mathbb{N}^{m}$, where the index $\alpha$ is an element of $\mathbb{Z}/m\mathbb{Z}$.
	We see that
	\begin{equation*}
		\begin{aligned}
			l_f^r(s)=
			&\sum_{
				\substack{
					n\text{ square-free}\\
					\omega(n)=r
				}
			}\prod_{p\mid n}f(p)p^{-s}
			\\
			=&\sum_{(r_{\alpha})_{\alpha}}
			\prod_{\alpha\in\mathbb{Z}/m\mathbb{Z}}
			f(\alpha)^{r_{\alpha}}
			l^{r_\alpha}_m(\alpha,s)
		\end{aligned}		
	\end{equation*}
	where $(r_{\alpha})_{\alpha}$ runs over points in $\mathbb{N}^{m}$ such that
	\begin{equation*}
		\sum_{\alpha\in\mathbb{Z}/m\mathbb{Z}}r_{\alpha}=r.
	\end{equation*}
	This expression first holds formally.
	However since the coefficients of $l_f^r(s)$ are bounded, it is absolutely and uniformly convergent in the half-plane $\sigma>1+\delta$ for any $\delta>0$.
	So it holds as an equation of holomorphic functions in the half-plane $\sigma>1$.
	And the properties of $l_f^r(s)$ just follow from those of $l_{m}^r(\alpha,s)$.
\end{proof}
\subsection{Counting homomorphisms}
In this section, let's apply the analytic tools to prove a result that will be used later.

Let $G$ be a finite abelian group and let $K$ be a $G$-field, i.e., $K/\mathbb{Q}$ is a Galois extension such that $G(K/\mathbb{Q})\cong G$.
Let $A$ be a finite $G$-module.
Denote the group of id{\`e}les of $K$ by $J_K$, and let $S_{\infty}$ be the set of primes of $K$ at infinity.
Denote the set of $G$-morphisms $\operatorname{Hom}_G(J_K^{S_\infty},A)$ by $\mathcal{T}_K$.
For each rational prime $p$, let $\mathscr{O}_p:=\mathscr{O}_K\otimes\mathbb{Z}_p=\prod_{\mathfrak{p}\mid p}\mathscr{O}_{\mathfrak{p}}$, where $\mathscr{O}_K$ is the ring of integers of $K$.
When $p=\infty$, just let $\mathscr{O}_{\infty}:=K\otimes\mathbb{R}$.
\begin{definition}\label{def: product of ramified primes for morphisms}
	Let $c:\operatorname{Hom}_{G}(\mathscr{O}_p^*,A)\to\mathbb{N}$ be a weight function for each prime $p$ (including the infinite prime) defined as follows:
	\begin{equation*}
		c(\rho_p)=\left\{
		\begin{aligned}
			&1\quad\text{if }p\nmid P(K)\infty\text{ and }\operatorname{im}(\rho_p)\neq\{0\}\\
			&0\quad\text{otherwise.}
		\end{aligned}
		\right.
	\end{equation*}
	Let $P$ be the counting function defined for $\mathcal{T}_K$ via
	\begin{equation*}
		P(\rho):=\prod_{p}p^{c(\rho_p)},
	\end{equation*}
	where the value of $p=\infty$ is taken to be the real constant $e$.
	We here just call it the product of ramified primes of $\rho$ for simplicity.
\end{definition}
For each rational prime $p$, let $\mathcal{T}_{p}$ be a nonempty subset of $\operatorname{Hom}_{G}(\mathscr{O}_p^*,G)\backslash\{1\}$, where $1$ means the trivial map.
In other words, this is a nontrivial local condition for each prime $p$.
Denote its complement $\operatorname{Hom}_{G}(\mathscr{O}_p^*,G)-\mathcal{T}_{p}$ as $\mathcal{T}_p^c$.
Let $m$ be a fixed integer such that $\lvert A\rvert\mid m$.
For each non-negative integer $r$, define $\mathcal{T}_{K}^r$ to be the subset of $\mathcal{T}_K$ such that a $G$-morphisms $\rho$ is contained in $\mathcal{T}_{K}^r$ if and only if there are exact $r$ primes $p\nmid m$ such that $\rho_p\in\mathcal{T}_p$.
Now let's first give an Euler product to estimate the number of homomorphisms in $\mathcal{T}_K^r$ ordered by the product of ramified primes $P$.
\begin{proposition}\label{prop: counting homomorphisms}
	For each non-negative integer $r$, define the generating series
	\begin{equation*}
		F_{\mathcal{T}_{K}^r,P}(s):=\sum_{\rho\in\mathcal{T}_K^r}P(\rho)^{-s},
	\end{equation*}
	and let
	\begin{equation}\label{eqn: estimate for F(r,s)}
		E_{\mathcal{T}_{K}^r,P}(s):=\prod_{p}\bigl(\sum_{\rho\in\mathcal{T}_p^c}p^{-c(\rho)s}\bigr)
		\sum_{
			\substack{
				n\text{ square-free}\\
				\gcd(m,n)=1,\omega(n)=r
			}
		}\prod_{p\mid n}\bigl(\sum_{\rho\in\mathcal{T}_p}p^{-c(\rho)s}\bigr).
	\end{equation}
	Write $F_{\mathcal{T}_{K}^r,P}(s)$ and $E_{\mathcal{T}_{K}^r,P}(s)$ in the form of Dirichlet series:
	\begin{equation*}
		F_{\mathcal{T}_{K}^r,P}(s)=\sum_{n=1}^{\infty}a_nn^{-s}
		\quad\text{and}\quad
		E_{\mathcal{T}_{K}^r,P}(s)=\sum_{n=1}^\infty b_nn^{-s}.
	\end{equation*}
	For each positive integer $n$, we have that
	\begin{equation*}
		a_n\leq b_n.
	\end{equation*}
\end{proposition}
\begin{proof}
	For a fixed integer $n$, let $\mathcal{T}_K(n)$ be the subset of $\mathcal{T}_K$ such that a $G$-morphism $\rho$ is contained in $\mathcal{T}_K(n)$ if and only if the condition
	\begin{equation*}
		\rho_p\in\mathcal{T}_p\iff p\mid n
	\end{equation*}
	holds.
	Define the generating series for $\mathcal{T}_K(n)$ as
	\begin{equation*}
		F_{\mathcal{T}_{K}(n),P}(s):=\sum_{\rho\in\mathcal{T}_K(n)}P(\rho)^{-s}.
	\end{equation*}
	If we write every generating series in the form of a Dirichlet series, then we have the identity of series:
	\begin{equation*}
		F_{\mathcal{T}_{K}^r,P}(s)=
		\sum_{
			\substack{
				n\text{ square-free}\\
				\gcd(m,n)=1,\omega(n)=r
			}
		}F_{\mathcal{T}_{K}(n),P}(s).
	\end{equation*}
	This is a corollary of the fact that
	\begin{equation*}
		\mathcal{T}_{K}^r=\bigsqcup_{
			\substack{
				n\text{ square-free}\\
				\gcd(m,n)=1,\omega(n)=r
			}
		}\mathcal{T}_{K}(n).
	\end{equation*}
	Note that if $\gcd(m,n)=1$, then the series $F_{\mathcal{T}_{K}(n),P}(s)$ admits a Euler product formally:
	\begin{equation*}
		F_{\mathcal{T}_{K}(n),P}(s)=
		\prod_{p\nmid n}\bigl(\sum_{\rho\in\mathcal{T}^c_p}p^{-c(\rho)s}\bigr)
		\prod_{p\mid n}\bigl(\sum_{\rho\in\mathcal{T}_p}p^{-c(\rho)s}\bigr).
	\end{equation*}
	This is a corollary of the fact that
	\begin{equation*}
		\operatorname{Hom}_{G}(J_K^{S_\infty},A)=\bigoplus_{p}\operatorname{Hom}_G(\mathscr{O}_p^*,A).
	\end{equation*}
	And we can rewrite $F_{\mathcal{T}_{K}^r,P}(s)$ (formally) as follows:
	\begin{equation}\label{eqn: counting homomorphisms 1}
		F_{\mathcal{T}_{K}^r,P}(s)=\sum_{
			\substack{
				n\text{ square-free}\\
				\gcd(m,n)=1,\omega(n)=r
			}
		}
		\Bigg[\prod_{p\mid n}\bigl(\sum_{\rho\in\mathcal{T}_p}p^{-c(\rho)s}\bigr)
		\cdot
		\prod_{p\nmid n}\bigl(\sum_{\rho\in\mathcal{T}_p^c}p^{-c(\rho)s}\bigr)\Bigg].
	\end{equation}
	Now let's consider the following series
	\begin{equation}\label{eqn: counting homomorphisms 2}
		\begin{aligned}
			E_{\mathcal{T}_{K}^r,P}(s)
			=&\prod_{p}\bigl(\sum_{\rho\in\mathcal{T}_p^c}p^{-c(\rho)s}\bigr)
			\sum_{
				\substack{
					n\text{ square-free}\\
					\gcd(m,n)=1,\omega(n)=r
				}
			}\prod_{p\mid n}\bigl(\sum_{\rho\in\mathcal{T}_p}p^{-c(\rho)s}\bigr)
			\\
			=&\sum_{
				\substack{
					n\text{ square-free}\\
					\gcd(m,n)=1,\omega(n)=r
				}
			}\prod_{p\mid n}\bigl(\sum_{\rho\in\mathcal{T}_p}p^{-c(\rho)s}\bigr)
			\cdot\prod_{p}\bigl(\sum_{\rho\in\mathcal{T}_p^c}p^{-c(\rho)s}\bigr)
			,
		\end{aligned}		
	\end{equation}
	defined as in the statement of the proposition.
	Note that we can switch the order of the operations for the identities hold formally (omitting the issue of convergence).
	Write $E_{\mathcal{T}_{K}^r,P}(s)=\sum_{n=1}^{\infty}b_nn^{-s}$.
	By the comparison between (\ref{eqn: counting homomorphisms 1}) and (\ref{eqn: counting homomorphisms 2}), we see that
	\begin{equation*}
		a_n\leq b_n
	\end{equation*}
	for all $n\geq1$.	
\end{proof}
Then let's give the asymptotic behaviour of $\sum_{n<X}b_n$ under some assumptions on the local specifications, where $E_{\mathcal{T}_{K}^r,P}(s)=\sum_{n=1}^\infty b_nn^{-s}$ as in (\ref{eqn: estimate for F(r,s)}).
\begin{proposition}\label{prop: estimate of counting homomorphisms}
	Let $M$ be a fixed nonzero integer, and let $f_1,f_2:\mathbb{N}\to\mathbb{N}$ be two maps such that for almost all primes $p\equiv q\bmod{M}$, we have $f_i(p)=f_i(q)$ with $i=1,2$, and $f_1$ is nonzero for infinitely many primes.
	We treat $f_1$ also as a function defined on $(\mathbb{Z}/M\mathbb{Z})^*$.
	Define
	\begin{equation}\label{eqn: E(f1,f2,s)}
		E_{P}(f_1,f_2,s):=\prod_{p}(1+f_2(p)p^{-s})
		\sum_{
			\substack{
				n\text{ square-free}\\
				\gcd(M,n)=1,\omega(n)=r
			}
		}\prod_{p\mid n}(1+f_2(p)p^{-s})
	\end{equation}
	and write it as a Dirichlet series $E_{P}(f_1,f_2,s)=\sum_{n=1}^\infty b_nn^{-s}$, then
	\begin{equation*}
		\sum_{n<X}b_n\ll X(\log X)^{\beta-1}(\log\log X)^r,
	\end{equation*}
	as $X\to\infty$, where $\beta$ is given by the formula
	\begin{equation*}
		\beta=\frac{1}{\phi(M)}\sum_{a\in(\mathbb{Z}/M\mathbb{Z})^*}f_1(a).
	\end{equation*}	
\end{proposition}
\begin{proof}
	Note that $f_1$ is nonzero on $(\mathbb{Z}/M\mathbb{Z})^*$ by our assumption.
	And the asymptotic behaviour of $\sum_{n<X}b_n$ is just a corollary of Lemma~\ref{lemma: zeta((q,a);s)} and Proposition~\ref{prop: count (q,a) primes} and the Tauberian Theorem~\ref{thm:Delange-Ikehara}.
\end{proof}
Roughly speaking, if the local conditions $\{\mathcal{T}_p\}_p$ satisfy some arithmetic property, then we obtain the estimate of the summatory $N_{\mathcal{T}_{K}^r,P}(X)$ in a relatively simple way.

\section{Abelian extensions}\label{section: abelian extensions}
In this section, let \(G\) be a finite abelian group.
The action of $G$ on itself by (left) multiplication turns it into a transitive permutation group in $S_m$ where $m=\lvert G\rvert$.
So \(\mathcal{S}:=\mathcal{S}(G)\) is just the set of abelian \(G\)-fields.
See also Definition~\ref{def:Gamma-fields} and~\ref{def:set of fields}.
The main result of this section is the following.
\begin{theorem}\label{thm:counting abelian fields}
	Let \(1\notin\Gamma\) be a non-empty subset of \(G\) that is closed under invertible powering.
	Let
	\begin{equation*}
		\beta:=\sum_{
			\substack{
				g\in G\backslash\Gamma\\
				g\neq1
			}
		}[\mathbb{Q}(\zeta_{r(g)}):\mathbb{Q}],			
	\end{equation*}
	where $r(g)$ is the order of $g$.
	For each non-negative integer $r$, we have
	\begin{equation*}
		N_{\mathcal{S}_{\Gamma}^r,P}(X)\ll X(\log X)^{\beta-1}(\log\log X)^{r},
	\end{equation*}
	as $X\to\infty$.
\end{theorem}
See Definition~\ref{def: specifactions} for the notation $\mathcal{S}_{\Gamma}^r$.
From the theorem, we immediately obtain the following.
\begin{corollary}
	Let $1\notin\Gamma$ be any nonempty subset of $G$ that is closed under invertible powering.
	The Hypothesis~\ref{hypothesis main} is true for \((\mathcal{S},P,\Gamma)\).
\end{corollary}
\begin{proof}
	Counting abelian fields $N_{\mathcal{S},P}(X)$ ordered by product of ramified primes is given by the formula
	\begin{equation*}
		N_{\mathcal{S},P}(X)\sim cX(\log X)^{\alpha-1}
	\end{equation*}
	where $c>0$ is a constant and
	\begin{equation*}
		\alpha=\sum_{1\neq g\in G}[\mathbb{Q}(\zeta_{r(g)}):\mathbb{Q}].
	\end{equation*}
	One can check Wood~\cite{wood2010probabilities} for more details.
	Since $\Gamma$ is nonempty and contains elements other than $1$, it is clear that $\beta<\alpha$, where 
	\begin{equation*}
		\beta:=\sum_{
			\substack{
				g\in G\backslash\Gamma\\
				g\neq1
			}
		}[\mathbb{Q}(\zeta_{r(g)}):\mathbb{Q}]
	\end{equation*}
	as in Theorem~\ref{thm:counting abelian fields}.
	So the corollary follows immediately.
\end{proof}
Also we get the desired result on distribution of class groups.
\begin{proof}[Proof of Theorem~\ref{thm:abelian case relative class groups}]
	Recall that $H$ is some subgroup of $G$, and $q$ is a rational prime such that $q^l\mid\lvert G/H\rvert$ with $l\geq1$.
	Apply the above corollary to
	\begin{equation*}
		\Omega_{q^l}:=\{g\in G\mid e(g)\equiv0\bmod{q^l}\},
	\end{equation*}
	and we see that the Hypothesis~\ref{hypothesis main} holds for $(\mathcal{S},P,\Omega_{q^l})$, and Theorem~\ref{thm:abelian case relative class groups} just follows Theorem~\ref{thm:P(rk X=r)=0} and~\ref{thm:infinite moment for bad primes}.
\end{proof}
In the rest of this section, let's prove Theorem~\ref{thm:counting abelian fields}.
Using Class Field Theory, we know that there exists a one-to-one correspondence
\begin{equation*}
	\operatorname{Sur}(\operatorname{C}_{\mathbb{Q}},G)\leftrightarrow\operatorname{Sur}(\prod_{p\nmid\infty}\mathbb{Z}_p^*,G).
\end{equation*}
Here $\operatorname{C}_K$ means the id{\`e}les class group of the number field $K$.
Recall from Definition~\ref{def: product of ramified primes for morphisms} that we have a counting function $P$ for $\rho:\prod_{p\nmid\infty}\mathbb{Z}_p^*\to G$.
In particular, if $\chi:\operatorname{C}_{\mathbb{Q}}\to G$ is a continuous homomorphism, just define
\begin{equation*}
	P(\chi):=P(\prod_{p\nmid\infty}\chi_p).
\end{equation*}
Moreover, if $\chi$ corresponds to a class field $K$, then
\begin{equation*}
	P(K)=P(\chi).
\end{equation*}
So in this case, the product of ramified primes of a homomorphism $\rho$ is consistent with the one for fields.
Now that the homomorphisms are equipped with a counting function, we can proving the theorem of this section.
\begin{proof}[Proof of Theorem~\ref{thm:counting abelian fields}.]
	Write $m:=\lvert G\rvert$.
	If $p\nmid m$, then the weight function $c(\rho_p)$ (see Definition~\ref{def: product of ramified primes for morphisms}) is totally determined by $\rho_p(\zeta_p)$ where $\zeta_p$ here means the generator of the group of roots of unity in $\mathbb{Z}_p^*$.
	So for each rational prime $p$, we can give the following local condition in the form of a set
	\begin{equation*}
		\mathcal{T}_p:=\{\rho:\mathbb{Z}_p^*\to G\mid \rho(\zeta_p)\in\Gamma\}.
	\end{equation*}
	Let $\mathcal{T}_p^c$ be the complement $\operatorname{Hom}(\mathbb{Z}_p^*,G)-\mathcal{T}_p$, and let
	\begin{equation*}
		E_{\mathcal{T}_{\Gamma}^r,P}(s):=
		\prod_{p}\bigl(\sum_{\rho\in\mathcal{T}_p^c}p^{-c(\rho)s}\bigr)
		\sum_{
			\substack{
				n\text{ square-free}\\
				\gcd(m,n)=1,\omega(n)=r
			}
		}\prod_{p\mid n}\bigl(\sum_{\rho\in\mathcal{T}_p}p^{-c(\rho)s}\bigr),
	\end{equation*}
	be the series defined as in (\ref{eqn: estimate for F(r,s)}).
	If we write it as a Dirichlet series $E_{\mathcal{T}_{\Gamma}^r,P}(s)=\sum_{n=1}^\infty b_nn^{-s}$, then by Proposition~\ref{prop: counting homomorphisms}, for all $X>0$, we have
	\begin{equation*}
		N_{\mathcal{S}_{\Gamma}^r,P}(X)\leq\sum_{n<X}b_n.
	\end{equation*}
	It is clear that the size of $\mathcal{T}_p$ is totally determined by $p$ mod $m$.
	In other words, if we define
	\begin{equation*}
		f_1(p):=\#\{\rho\in\mathcal{T}_p\mid c(\rho)=1\}
		\quad\text{and}\quad
		f_2(p):=\#\{\rho\notin\mathcal{T}_p\mid c(\rho)=1\}
	\end{equation*}
	then $f_i(p)=f_i(q)$ whenever $p\equiv q\bmod{m}$, where $i=1,2$.	
	And we see that
	\begin{equation*}
		E_{\mathcal{T}_{\Gamma}^r,P}(s)=E_P(f_1,f_2,s)
	\end{equation*}
	as defined in (\ref{eqn: E(f1,f2,s)}).
	The conditions of Proposition~\ref{prop: estimate of counting homomorphisms} are all satisfied, and we have
	\begin{equation*}
		N_{\mathcal{S}_{\Gamma}^r,P}(X)\ll X(\log X)^{\beta'-1}(\log\log X)^r.
	\end{equation*} 
	where
	\begin{equation*}
		\beta'=\frac{1}{\phi(m)}\sum_{a\in(\mathbb{Z}/m\mathbb{Z})^*}f_1(a).
	\end{equation*}
	It suffices to show that $\beta'=\beta$.
	For each $n\mid m$, let $\pi_{n}:(\mathbb{Z}/m\mathbb{Z})^*\to(\mathbb{Z}/n\mathbb{Z})^*$ be the surjective homomorphism between multiplicative groups.
	Moreover we view $f_1$ as a function defined on $(\mathbb{Z}/m\mathbb{Z})^*$.
	Then we have the following computation:
	\begin{equation*}
		\begin{aligned}
			\beta'=&
			\frac{1}{\phi(m)}\sum_{a\in(\mathbb{Z}/m\mathbb{Z})^*}f_{1}(a)
			\\
			=&\frac{1}{\phi(m)}
			\sum_{a\in(\mathbb{Z}/m\mathbb{Z})^*}
			\#\{1\neq g\in G\backslash\Gamma\mid\pi_{r(g)}(a)=1\}
			\\
			=&\frac{1}{\phi(m)}
			\sum_{1\neq g\in G\backslash\Gamma}\#\pi_{r(g)}^{-1}(1)
			\\
			=&\frac{1}{\phi(m)}\sum_{1\neq g\in G\backslash\Gamma}\frac{\phi(m)}{\phi(r(g))}
			\\
			=&\sum_{1\neq g\in G\backslash\Gamma}\phi(r(g))^{-1}=\beta.
		\end{aligned}
	\end{equation*}
	And we are done for the proof.
\end{proof}

\section{\texorpdfstring{$D_4$}{D4} extensions}\label{section:D4 extensions}
In this section, let $D_4=\langle\sigma,\tau\mid\sigma^4=1=\tau^2,\tau^{-1}\sigma\tau=\sigma^3\rangle$ be the dihedral group of order $8$.
Define $\mathcal{S}:=\mathcal{S}(D_4)$ with $D_4$ being a subgroup of $S_4$, i.e., we are focused on the quartic extensions whose Galois closure are $D_4$-fields (see also Definition~\ref{def:Gamma-fields} and~\ref{def:set of fields}). 
\subsection{Main results}
We first introduce the definition of the \emph{conductor} for a quadratic extension of a quadratic number field, which will be used here as the counting function of the number fields.
\begin{definition}[Conductor]
	If \(K\) is a quadratic number field and \(L\) is a quadratic extension of \(K\),
	define the \emph{conductor} of the pair \((L,K)\) as
	\[C(L,K):=\frac{\operatorname{Disc}(L)}{\operatorname{Disc}(K)}.\]
	If \(L\) is a \(D_4\)-field and \(K\) denotes its (unique) quadratic subfield, then \(C(L,K)=C(L)\) (the conductor of \(L\)).
\end{definition}
Note that the notation given by the above definition agrees with the Artin conductor for the irreducible $2$-dimensional representation of $D_4$, if the quartic field has $D_4$-Galois closure.
See also \cite[Section 2.3]{altug2017number} for details.
We here follow this definition for the convenience of both computation and generalization to other quartic fields.
Define the following notation:
\begin{equation*}
	H_{\mathcal{S},C}(X):=\sum_{\substack{
			L\in\mathcal{S}\\
			C(L)<X
	}}\lvert\operatorname{Hom}(\operatorname{Cl}_L,C_2)\rvert.
\end{equation*}
In this section we are going to prove the following main result.
\begin{theorem}\label{thm:D4 case conductor counting fields}
	We have
	\[\mathbb{E}_{\mathcal{S},C}(\lvert\operatorname{Hom}(\operatorname{Cl}_L,C_2)\rvert)=\lim_{X\to\infty}\frac{H_{\mathcal{S},C}(X)}{N_{\mathcal{S},C}(X)}=+\infty,\]
	where $C$ is the conductor defined above.
\end{theorem}
\subsection{Preliminary tools}
We want to study the statistical behaviour of $\operatorname{Cl}_L[2^\infty]$ where $L\in\mathcal{S}$ for the conductor.
Let's first prove a lemma whose statement is similar to \cite[Lemma 5.1]{altug2017number}.
Recall that \(\omega(n)\) means the number of distinct prime factors dividing \(n\), and $\mu(n)$ is the M{\"o}bius function, and $\left(\dfrac{\cdot}{\cdot}\right)$ is the Legendre symbol.
\begin{lemma}\label{lemma:D4 conductor}
	For any $0<\epsilon<\frac{1}{2}$, we have
	\begin{equation*}
		\sum_{
			\substack{
				0<D<X\\
				D\text{ squarefree}
			}
		}
		\frac{2^{\omega(D)}}{D}
		\cdot
		\Biggl(
		\sum_{m=1}^\infty
		\sum_{
			\substack{
				n=1\\mn\neq\square
			}
		}^{D^{\frac{1}{2}+\epsilon}}
		\frac{\mu(m)}{m^2n}\Bigl(\frac{D}{mn}\Bigr)
		\Biggr)=o(X),
	\end{equation*}
	as $X\to\infty$.
\end{lemma}
\begin{proof}
	The proof itself is straightforward. 
	First of all, by comparing with $\zeta(s)^2$, we see that the Euler product
	\[\prod_{p\text{ rational prime}}(1+2p^{-s})\]
	admits a meromorphic continuation to the closed half-plane $\Re(s)\geq1$ with a unique pole of order $2$ at $s=1$.
	Then, according to \cite[Theorem 5.11]{montgomery2006multiplicative}, we have
	\begin{equation*}
		\sum_{n<X}\frac{1}{n}\sim c_1\log X
		\quad\text{and}\quad
		\sum_{
			\substack{
				0<D<X\\
				D\text{ squarefree}
			}
		}
		\frac{2^{\omega(D)}}{D}\sim c_2(\log X)^2
	\end{equation*}
	for some constant $c_1,c_2>0$.
	So, we know that
	\[\begin{aligned}
		&\Big\lvert
		\sum_{\substack{
				0<D<X\\
				D\text{ squarefree}
		}}
		\frac{2^{\omega(D)}}{D}\cdot
		\Bigl(
		\sum_{m=1}^{\infty}
		\sum^{D^{\frac{1}{2}+\epsilon}}_{\substack{
				n=1\\
				mn\neq\square
		}}\frac{\mu(m)}{m^2n}\bigl(\frac{D}{mn}\bigr)
		\Bigr)
		\Big\rvert
		\\
		\leq&\sum_{\substack{
				0<D<X\\
				D\text{ squarefree}
		}}
		\frac{2^{\omega(D)}}{D}\cdot
		\Bigl(
		\sum_{m=1}^\infty\sum_{n=1}^{X}\frac{1}{m^2n}
		\Bigr)
		\\
		\ll&(\log X)^3=o(X),\text{ as }X\to\infty.
	\end{aligned}\]
\end{proof}
Using the above lemma, we can prove the following proposition, which is similar to \cite[Proposition 5.2]{altug2017number}.
\begin{proposition}\label{prop:D4 conductor}
	We have:
	\begin{equation}\label{eqn:prop D4 conductor}
		\begin{aligned}
			&\sum_{
				\substack{
					[K:\mathbb{Q}]=2\\
					0<\operatorname{Disc}(K)<X
				}
			}
			\frac{L(1,K/\mathbb{Q})}{L(2,K/\mathbb{Q})}\cdot\frac{2^{\omega(\operatorname{Disc}(K))}}{\operatorname{Disc}(K)}
			\\
			=&\sum_{
				\substack{
					[K:\mathbb{Q}]=2\\
					0<\operatorname{Disc}(K)<X
				}
			}
			\frac{2^{\omega(\operatorname{Disc}(K))}}{\operatorname{Disc}(K)}
			\cdot
			\sum_{
				\substack{
					0<a,b<\infty\\
					(\operatorname{Disc}(K),ab)=1
				}
			}\frac{\mu(a)}{a^3b^2}+o(X),
			\\
			\text{and }&\sum_{
				\substack{
					[K:\mathbb{Q}]=2\\
					-X<\operatorname{Disc}(K)<0
				}
			}\frac{L(1,K/\mathbb{Q})}{L(2,K/\mathbb{Q})}\cdot\frac{2^{\omega(\operatorname{Disc}(K))}}{\operatorname{Disc}(K)}
			\\
			=&\sum_{
				\substack{
					[K:\mathbb{Q}]=2\\
					-X<\operatorname{Disc}(K)<0
				}
			}
			\frac{2^{\omega(\operatorname{Disc}(K))}}{\operatorname{Disc}(K)}
			\cdot
			\sum_{\substack{
					0<a,b<\infty\\
					(\operatorname{Disc}(K),ab)=1
				}
			}\frac{\mu(a)}{a^3b^2}+o(X),
		\end{aligned}
	\end{equation}
	as $X\to\infty$, where $L(s,K/\mathbb{Q})=\sum_n\frac{\chi_K(n)}{n^s}$ and $\chi_K$ is the quadratic character associated to $K$.
\end{proposition}
\begin{proof}
	The proof is similar to that of \cite[Proposition 5.2]{altug2017number}.
	According to \cite[(17)]{altug2017number}, we have
	\begin{equation*}
		\frac{L(1,K/\mathbb{Q})}
		{L(2,K/\mathbb{Q})}
		=\frac{1}
		{L(2,K/\mathbb{Q})}
		\cdot
		\sum_{n=1}^{\lvert\operatorname{Disc}(K)\rvert^{\frac{1}{2}+\epsilon}}
		\frac{\chi_K(n)}
		{n}
		+O_\epsilon\Biggl(
		\frac{\log(\lvert\operatorname{Disc}(K)\rvert)}
		{\lvert\operatorname{Disc}(K)\rvert^\epsilon}
		\Biggr).
	\end{equation*}
	Let's rewrite the left-hand sides of (\ref{eqn:prop D4 conductor}) as
	\begin{equation}\label{eqn:prop D4 conductor 1}
		\sum_{
			\substack{
				[K:\mathbb{Q}]=2\\
				0<\operatorname{Disc}(K)<X
			}
		}\frac{2^{\omega(\operatorname{Disc}(K))}}{\operatorname{Disc}(K)}\cdot
		\Biggl(
		\frac{1}
		{L(2,K/\mathbb{Q})}
		\cdot\sum_{n=1}^{\lvert\operatorname{Disc}(K)\rvert^{\frac{1}{2}+\epsilon}}
		\frac{\chi_K(n)}
		{n}
		+O_\epsilon\Bigl(
		\frac{\log(\lvert\operatorname{Disc}(K)\vert)}
		{\lvert\operatorname{Disc}(K)\rvert^\epsilon}
		\Bigr)
		\Biggr),
	\end{equation}
	and similar expression for imaginary quadratic fields.
	Note that
	\begin{equation*}
		\sum_{
			\substack{
				[K:\mathbb{Q}]=2\\
				0<\lvert\operatorname{Disc}(K)\rvert<X
			}
		}
		\frac{2^{\omega(\operatorname{Disc}(K))}}
		{\operatorname{Disc}(K)}
		\frac{\log(\lvert\operatorname{Disc}(K)\rvert)}
		{\lvert\operatorname{Disc}(K)\rvert^\epsilon}
		=O_\epsilon(1),
	\end{equation*}
	we focus on the remaining terms.
	Then by \cite[(20)]{altug2017number}, we know that
	\begin{equation}\label{eqn:prop D4 conductor 2}
		\begin{aligned}
			&\frac{1}
			{L(2,K/\mathbb{Q})}
			\sum_{n=1}^{\lvert\operatorname{Disc}(K)\rvert^{\frac{1}{2}+\epsilon}}
			\frac{\chi_K(n)}{n}
			\\
			=&\sum_{
				\substack{
					0<a,b<\infty\\
					(\operatorname{Disc}(K),ab)=1
				}
			}
			\frac{\mu(a)}
			{a^3b^2}
			+\sum_{n=1}^{\lvert\operatorname{Disc}(K)\rvert^{\frac{1}{2}+\epsilon}}
			\frac{\chi_K(n)}
			{n}
			\sum^{\infty}_{
				\substack{
					m=1\\
					mn\neq\square
				}
			}
			\frac{\mu(m)\chi_K(m)}
			{m^2}.
		\end{aligned}
	\end{equation}
	By substituting (\ref{eqn:prop D4 conductor 2}) back to (\ref{eqn:prop D4 conductor 1}), we have
	\begin{equation}\label{eqn:prop D4 conductor 3}
		\sum_{
			\substack{
				[K:\mathbb{Q}]=2\\
				0<\operatorname{Disc}(K)<X
			}
		}
		\frac{2^{\omega(\operatorname{Disc}(K))}}
		{\operatorname{Disc}(K)}
		\Biggl(
		\sum_{
			\substack{
				0<a,b<\infty\\
				(\operatorname{Disc}(K),ab)=1
			}
		}
		\frac{\mu(a)}
		{a^3b^2}
		+\sum_{n=1}^{\lvert\operatorname{Disc}(K)\rvert^{\frac{1}{2}+\epsilon}}
		\frac{\chi_K(n)}
		{n}
		\sum^{\infty}_{
			\substack{
				m=1\\
				mn\neq\square
			}
		}
		\frac{\mu(m)\chi_K(m)}
		{m^2}
		\Biggr),
	\end{equation}
	and similar expression for imaginary quadratic fields.
	By Lemma~\ref{lemma:D4 conductor}, we know that
	\begin{equation*}
		\sum_{
			\substack{
				[K:\mathbb{Q}]=2\\
				0<\operatorname{Disc}(K)<X
			}
		}
		\frac{2^{\omega(\operatorname{Disc}(K))}}
		{\operatorname{Disc}(K)}
		\sum_{n=1}^{\lvert\operatorname{Disc}(K)\rvert^{\frac{1}{2}+\epsilon}}
		\frac{\chi_K(n)}
		{n}
		\sum^{\infty}_{
			\substack{
				m=1\\
				mn\neq\square
			}
		}
		\frac{\mu(m)\chi_K(m)}
		{m^2}
		=o(X).
	\end{equation*}
	Similarly for the case when $\operatorname{Disc}(K)<0$.
	And we are done for the proof.
\end{proof}
\subsection{Proof of Theorem~\ref{thm:D4 case conductor counting fields}}
Let's prove the theorem in this section.
We first need a summatory function with two variables.
In this section, let $d_K$ be the absolute discriminant of the quadratic field $K$.
Define
\begin{equation*}
	H_C(X,Y):=
	\sum_{
		\substack{
			L\in\mathcal{S},C(L)<X\\
			d_K<Y
		}
	}2^{\omega(P(K))},
\end{equation*}
where $\omega(n)$ is the number of different prime factors of $n$, and $K$ is the unique quadratic subfield of $L$.
Following the same idea, let
\begin{equation*}
	H_C(X):=\sum_{L\in\mathcal{S},C(L)<X}2^{\omega(P(K))}.
\end{equation*}
See also \cite[Theorem 4.3]{altug2017number}.
\begin{proof}[Proof of Theorem~\ref{thm:D4 case conductor counting fields}.]
	The basic idea of the proof is similar to that of \cite[Theorem 2]{altug2017number}.
	Because of Theorem~\ref{thm:S1 invariant part}, it suffices to show that as $X\to\infty$ we have that
	\begin{equation*}
		N_{\mathcal{S},C}(X)=o(H_C(X)).
	\end{equation*}
	Because Theorem~\ref{thm:rz69} implies that 
	\begin{equation*}
		\lvert\operatorname{Hom}(\operatorname{Cl}_L,C_2)\rvert\geq 2^{\omega(P(K))-6}.
	\end{equation*}
	Second, let $\chi_{[0,1]}$ is the characteristic function of $[0,1]$.
	Then we can write
	\[\begin{aligned}
		H_C(X,X^\beta)=
		&\sum_{\substack{
				[K:\mathbb{Q}]=2\\
				d_K<X^\beta}}
			2^{\omega(P(K))}\cdot
		\sum_{\substack{
				[L:\mathbb{Q}]=2\\
				L\in\mathcal{S}}}
		\chi_{[0,1]}\left(\frac{\lvert\operatorname{Disc}(K)\operatorname{Nm}_{K/\mathbb{Q}}(\operatorname{Disc}(L/K))\rvert}{X}\right),
	\end{aligned}\]
	where $0<\beta<1$.
	For $\epsilon>0$, there exists some compactly supported smooth functions $\varphi^{\pm}:\mathbb{R}_{\geq0}\to\mathbb{R}_{\geq0}$ such that $\varphi^{\pm}-\chi_{[0,1]}$ takes value in $\mathbb{R}^{\pm}$ and that $\operatorname{Vol}(\varphi^{\pm})=1\pm\epsilon$.
	Then \cite[Lemma 4.5 and (13)]{altug2017number} implies that
	\[\begin{aligned}
		&\sum_{\substack{
				[K:\mathbb{Q}]=2\\
				d_K<X^\beta}}
			2^{\omega(P(K))}\cdot
		\sum_{\substack{
				[L:\mathbb{Q}]=2\\
				L\in\mathcal{S}}}
		\varphi^{\pm}
		\Bigl(
		\frac{d_K\lvert\operatorname{Nm}_{K/\mathbb{Q}}(\operatorname{Disc}(L/K))\rvert}{X}
		\Bigr)
		\\
		=&\sum_{\substack{
				[K:\mathbb{Q}]=2\\
				d_K<X^\beta}}
		\frac{1\pm\epsilon}{\zeta(2)}\cdot\frac{L(1,K/\mathbb{Q})}{L(2,K/\mathbb{Q})}\cdot\frac{2^{\omega(P(K))-r_2(K)}}{d_K}X
		\\
		+&O_\epsilon
		\Bigl(
		\sum_{\substack{
				[K:\mathbb{Q}]=2\\
				d_K<X^\beta}
			}
			2^{2\omega(P(K))}
		\lvert\operatorname{Disc}(K)\rvert^{-\frac{1}{4}+\epsilon}X^{\frac{1}{2}+\epsilon}
		\Bigr),
	\end{aligned}\]
	where $r_2(K)$ means the number of complex embeddings of $K$.
	Note that the error term is bounded by $O_\epsilon(X^{1/2+3\beta/4+\epsilon})$.
	So, when $\beta<2/3$, the error term is bounded by $o(X)$.
	By letting $\epsilon$ goes to $0$, it suffices to prove that
	\[\sum_{\substack{
			[K:\mathbb{Q}]=2\\\lvert\operatorname{Disc}(K)\rvert<Y}}
	\frac{L(1,K/\mathbb{Q})}{L(2,K/\mathbb{Q})}\frac{2^{\omega(P(K))}}{d_K}\sim c(\log Y)^2,\]
	where $c$ is a positive constant.
	Now Proposition~\ref{prop:D4 conductor} says that it suffices to prove that
	\begin{equation}\label{eqn:thm D4 case conductor counting fields}
		\begin{aligned}
			&\sum_{\substack{
					[K:\mathbb{Q}]=2\\
					0<\operatorname{Disc}(K)<Y}}
			\frac{L(1,K/\mathbb{Q})}{L(2,K/\mathbb{Q})}\cdot\frac{2^{\omega(P(K))}}{\operatorname{Disc}(K)}\\
			=&\sum_{\substack{
					[K:\mathbb{Q}]=2\\
					0<\operatorname{Disc}(K)<Y}}
			\frac{2^{\omega(P(K))}}{\operatorname{Disc}(K)}
			\cdot\sum_{\substack{
					0<a,b<\infty\\
					(\operatorname{Disc}(K),ab)=1}}
			\frac{\mu(a)}{a^3b^2}\sim c_1(\log Y)^2,
			\\
			\text{and }&\sum_{\substack{
					[K:\mathbb{Q}]=2\\
					-Y<\operatorname{Disc}(K)<0
			}}
			\frac{L(1,K/\mathbb{Q})}{L(2,K/\mathbb{Q})}\cdot\frac{f(L)}{\operatorname{Disc}(K)}
			\\
			=&\sum_{\substack{
					[K:\mathbb{Q}]=2\\
					-Y<\operatorname{Disc}(K)<0}}
			\frac{2^{\omega(P(K))}}{\operatorname{Disc}(K)}
			\cdot\sum_{\substack{
					0<a,b<\infty\\
					(\operatorname{Disc}(K),ab)=1}}
			\frac{\mu(a)}{a^3b^2}\sim c_2(\log Y)^2.
		\end{aligned}
	\end{equation}
	Using \cite[(21)]{altug2017number},
	\begin{equation*}
		\sum_{
			\substack{
				0<a,b<\infty\\
				(\operatorname{Disc}(K),ab)=1
			}
		}
		\frac{\mu(a)}
		{a^3b^2}
		=\frac{\zeta(2)}
		{\zeta(3)}
		\prod_{p\mid\operatorname{Disc}(K)}
		\frac{1-\frac{1}{p^2}}
		{1-\frac{1}{p^3}},
	\end{equation*}
	we get
	\begin{equation*}
		\lim_{Y\to\infty}
		\frac{\zeta(2)}
		{\zeta(3)\log(Y)^2}
		\sum_{
			\substack{
				1<D<Y\\
				D\text{ squarefree}
			}
		}
		\frac{2^{\omega(D)}}{D}
		\prod_{p\mid D}
		\frac{1-\frac{1}{p^2}}
		{1-\frac{1}{p^3}}
		=\zeta(2)\prod_{p\in\mathcal{P}}
		(1-\frac{3}{p^2}-\frac{1}{p^3}+\frac{6}{p^4}-\frac{3}{p^5}).
	\end{equation*}
	We can do the similar computation when the squarefree $D$ takes $D\equiv\pm1\bmod{4}$ and so on, i.e., discriminant of quadratic number fields belongs to different conjugacy classes (see \cite[(23)]{altug2017number} for example). 
	Similar result holds when $\operatorname{Disc}(K)<0$.
	This proves (\ref{eqn:thm D4 case conductor counting fields}), hence the theorem.
\end{proof}

\section{Dihedral extensions in general}\label{section:Dihedral extensions}
In this section, we focus on the case when the Galois group $G$ is a Dihedral group.
We use the following notation to explain what we mean by dihedral extensions in this section.
\begin{definition}
    Let \(G\) be a finite group.
    We call \(G\) a dihedral group if there exists a split short exact sequence
    \begin{equation*}
    	1\to C_m\to G\to C_2\to1,
    \end{equation*}
    of groups, where \(C_m\) is the cyclic group of order \(m\), and \(m\geq3\) is an integer, such that for the generator \(g_2\in C_2\) and for a generator \(g_m\in C_m\) we have \(g_2g_mg_2^{-1}=g_m^{-1}\).
    In this case we just write \(G=D_m\)
\end{definition}
In this section, let $m$ be a fixed integer$\geq3$.
Define 
\begin{equation*}
	\mathcal{S}:=\{L\in\mathcal{S}(D_m)\mid S^1\cap\hat{L}=\{\pm1\}\},
\end{equation*}
where $\hat{L}$ is the Galois closure of $L$.
See also Definition~\ref{def:Gamma-fields} and~\ref{def:set of fields}.
Roughly speaking, we only consider the fields whose group of roots of unity is the same as the base field.
For example, the closure $\hat{L}$ cannot contain $\mathbb{Q}(\sqrt{-1})$ or $\mathbb{Q}(\sqrt{-3})$ as a subfield.
We present our main result of this section as follows.
Recall that for an element $g\in G$, let $r(g)$ be its order.
\begin{theorem}\label{thm: estimate of Dq}	
	Let $m=q^l$, where $q$ is an odd prime and $l$ is a positive integer.
	Let \(\Omega=C_m\backslash\{1\}\) be the non-empty subset of \(D_m\).
	Assume that the Cohen-Lenstra Heuristics holds for quadratic number fields, 
	or to be precise, assume that as $X\to\infty$ we have
	\begin{equation*}
		\sum_{\substack{
				K_2=\mathbb{Q}(\sqrt{d})\\
				P(K_2)<X
		}}\lvert\operatorname{Hom}(\operatorname{Cl}_{K_2},C_m)\rvert\ll X.
	\end{equation*}
	Then for each \(r=1,2,\dots\), we have
	\begin{equation*}
		N_{\mathcal{S}_{\Omega}^r,P}(X)\ll X(\log\log X)^{r+1},
	\end{equation*}
	as $X\to\infty$.
\end{theorem}
When $m=4$, we are looking at $D_4$-fields, but this time they are ordered by the product of ramified primes.
And we can show the following on field-counting.
\begin{theorem}\label{thm: D4 counting fields P}
	Let $m=4$, and let $\Omega:=\{g_4,g_4^3\}$ where $g_4$ is the generator of $C_4$.
	Then for each non-negative integer $r$, we have that
	\begin{equation*}
		N_{\mathcal{S}_{\Omega}^r,P}(X)\ll X(\log X)^2(\log\log X)^{r+1}.
	\end{equation*}
\end{theorem}
We'll prove these two theorems later in this section.
Although in general, the moments of class groups are unknown even for quadratic fields, we still have some proven results.
And let's see an example.
\begin{example}
	Let $m=3$, the group $D_3$ is just $S_3$.
	And let $\Omega=\{(123),(132)\}$.
	The set $\mathcal{S}_{\Omega}^r$ means the set of non-Galois cubic fields such that its Galois closure contains only $\{\pm1\}$ as its group of roots of unity and such that it has exactly $r$ totally ramified primes other than $2$ and $3$.
	By Davenport and Heilbronn~\cite{Davenport1971Cubic}, we know that there exists some constant $c_3>0$ such that
	\begin{equation*}
		\mathbb{E}_{\mathcal{S}(C_2),P}(\lvert\operatorname{Hom}(\operatorname{Cl}_K,C_3)\rvert)=c_3,
	\end{equation*}
	thereby the conditions of Theorem~\ref{thm: estimate of Dq} are all satisfied.
	So as $X\to\infty$ we have
	\begin{equation*}
		N_{\mathcal{S}_{\Omega}^r,P}(X)\ll X(\log\log X)^{r+1}.
	\end{equation*}
\end{example}
One of the applications of Theorem~\ref{thm: estimate of Dq} is for the study of the distribution of class groups,
because we can verify the Hypothesis~\ref{hypothesis main} provided that the conditions of the theorem hold.
\begin{proposition}\label{prop:Dq-case}
	Let $m=q^l$, where $q$ is an odd prime and $l$ is a positive integer.
	Let \(\Omega=\{g\in C_m\mid g\neq1\}\).
	Assume that the the Cohen-Lenstra Heuristics hold for quadratic number fields, or to be precise, as $X\to\infty$, we have
	\[\sum_{\substack{
			K_2=\mathbb{Q}(\sqrt{d})\\
			P(K_2)<X
	}}\lvert\operatorname{Hom}(\operatorname{Cl}_{K_2},C_m)\rvert\ll X,\]
	and the Malle-Bhargava Heuristics hold for \(D_m\)-fields, or as $X\to\infty$, we have that
	\begin{equation*}
		N_{\mathcal{S},P}(x)\gg X\log X,
	\end{equation*}
	where
	\begin{equation*}
		\beta=\sum_{1\neq g\in C_m}[\mathbb{Q}(\zeta_{r(g)}):\mathbb{Q}].
	\end{equation*}
	Then the Hypothesis~\ref{hypothesis main} holds for $(\mathcal{S},P,\Omega)$.
	Also, for all $r=0,1,2,\dots$, we have
	\begin{equation*}
		\mathbb{P}_{\mathcal{S},P}(\operatorname{rk}_q\operatorname{Cl}_K\leq r)=0\quad\text{and}\quad\mathbb{E}_{\mathcal{S},P}(\lvert\operatorname{Hom}(\operatorname{Cl}_K,C_q)\rvert)=+\infty.
	\end{equation*}
\end{proposition}
\begin{proof}
	Theorem~\ref{thm: estimate of Dq} together with the hypothesis on $N_{\mathcal{S},P}(X)$ shows that the Hypothesis~\ref{hypothesis main} is true in this case.
	And the rest just follows from Theorem~\ref{thm:P(rk X=r)=0} and~\ref{thm:infinite moment for bad primes}.
\end{proof}
Similarly, when $m=4$, we have the following results on the higher torsion part of $\operatorname{Cl}_L$.
\begin{proposition}\label{prop: D4 dist of 2CL}
	Let $m=4$.
	If the Malle-Bhargava Heuristics is true for $\mathcal{S}$, or we have
	\begin{equation*}
		N_{\mathcal{S},P}(X)\gg X(\log X)^3
	\end{equation*}
	as $X\to\infty$, then for each non-negative integer $r$, we have 
	\begin{equation*}
		\mathbb{P}_{\mathcal{S},P}(\operatorname{rk}_{2}2\operatorname{Cl}_L\leq r)=0
	\end{equation*}
	where $L$ runs over fields in $\mathcal{S}$ ordered by $P$.
	In addition,
	\begin{equation*}
		\mathbb{E}_{\mathcal{S},P}(\lvert\operatorname{Hom}(2\operatorname{Cl}_L,C_2)\rvert)=+\infty.
	\end{equation*}
\end{proposition}
\begin{proof}
	Using Theorem~\ref{thm: D4 counting fields P} and the assumption from Malle-Bhargava Heuristics, we see that the Hypothesis~\ref{hypothesis main} holds for $(\mathcal{S},\{g_4,g_4^3\},P)$ where $g_4$ is the generator of $C_4$.
	So the rest just follows from Theorem~\ref{thm:P(rk X=r)=0} and~\ref{thm:infinite moment for bad primes}.
\end{proof}
Though these results require additional conditions deduced from the Malle-Bhargava Heuristics which is widely open for now,
it indicates that the structure of the class groups $\operatorname{Cl}_L[2^\infty]$ of $D_4$-fields admits nontrivial higher torsion part in the sense of statistics.

Let's start proving the two theorems above.
First we summarize the properties we need from Class Field Theory and group cohomology as the preliminaries.
For a number field $K$, if $\mathfrak{p}$ is a prime of $K$, then let $\mathscr{O}_{\mathfrak{p}}:=\{x\in K_{\mathfrak{p}}\mid\lvert x\rvert_{\mathfrak{p}}\leq1\}$.
And let $\mathscr{O}_{\mathfrak{p}}^*:=\{x\in K_{\mathfrak{p}}\mid\lvert x\rvert_{\mathfrak{p}}=1\}$ be the units of $\mathscr{O}_{\mathfrak{p}}$.
In addition, if $K$ is a $G$-field (Galois over $\mathbb{Q}$ with $G\cong G(K/\mathbb{Q})$), then for each rational prime $p$, define
\begin{equation*}
	\mathscr{O}_p:=\prod_{\mathfrak{p}\mid p}\mathscr{O}_{\mathfrak{p}}=\mathscr{O}_K\otimes\mathbb{Z}_p,
\end{equation*}
where $\mathscr{O}_K$ is the ring of integers of $K$.
Clearly, since $K$ is Galois with $G(K/\mathbb{Q})\cong G$, we see that $\mathscr{O}_p$ is also equipped with an action of $G$.
\begin{proposition}\label{prop: CFT and dihedral extensions}
	Let $G$ be a finite group and let $K$ be a $G$-field.
	\begin{enumerate}
		\item Let $A$ be a finite $G$-module.
		There is a one-to-one correspondence between
		\begin{equation*}
			\operatorname{Sur}_{G}(\operatorname{C}_K,A)\leftrightarrow\{L\in\mathcal{S}(A,K)\mid L\text{ is Galois over }\mathbb{Q}\text{ and }G(L/\mathbb{Q})\in\operatorname{EXT}(G,A)\}.
		\end{equation*}
		\item If $m$ is odd, equip $C_m$ with the $C_2$-action by $g_2\cdot h=h^{-1}$ where $g_2$ is the generator of $C_2$ and $h\in C_m$.
		Then
		\begin{equation*}
			H^2(C_2,C_m)=0.
		\end{equation*}
	\end{enumerate}
\end{proposition}
\begin{proof}[Sketch of the proof]
	(1):
	Note that Class Field Theory says that if $L/K$ is an abelian $A$-extension, then
	\begin{equation*}
		\operatorname{C}_K/\operatorname{Nm}_{L/K}\operatorname{C}_L\cong A.
	\end{equation*}
	In addition $L/\mathbb{Q}$ is Galois if and only if $\operatorname{Nm}_{L/K}\operatorname{C}_L$ is a well-defined $G$-module.
	The ``only if'' part is easy to see.
	The ``if'' part follows from the fact that if $L'$ is a conjugate of $L$ in the Galois closure, then the Galois conjugate will induce another $A$-extension
	\begin{equation*}
		\operatorname{C}_K/\operatorname{Nm}_{L'/K}\operatorname{C}_{L'}\cong A.
	\end{equation*}
	Since the isomorphism defines the action of $G$ on $A$, the rest just follows.
	
	(2):
	This is a corollary of the standard result:
	if $G$ is finite and $A$ is a $G$-module, then
	\begin{equation*}
		\lvert G\rvert\cdot H^n(G,A)=0
	\end{equation*}
	for all $n>0$.
	See \cite[I.1.6.1]{neukirch2013cohomology} for example.
\end{proof}
One of the reasons we need these preliminaries is the following corollary.
\begin{corollary}\label{cor: dihedral extensions and morphisms}
	Let $K$ be a fixed quadratic number field
	Let $C_m$ be the $C_2$-module defined by the action $g\cdot h=h^{-1}$ where $g$ is the generator of $C_2$ and $h\in C_m$.
	There is an injective map
	\begin{equation*}
		\{L\in\mathcal{S}\mid K=\hat{L}^{C_m}\}\rightarrow\operatorname{Sur}_{C_2}(\operatorname{C}_K,C_m),
	\end{equation*}
	where $\hat{L}$ means the Galois closure of $L$.
	In particular, if $m$ is odd, this is a one-to-one correspondence.
\end{corollary}
This follows from Proposition~\ref{prop: CFT and dihedral extensions}(1) and (2) directly.
It should be noted that because of our definition for $\mathcal{S}$, the quadratic number field $K$ is a subfield of some conjugate of $L$.
In other words we should call $K$ the associated quadratic number field of $L$.
In the sense of field-counting, this is of minor issue.

Because of this corollary, for a quadratic number field $K$, define
\begin{equation}\label{eqn: dihedral T(K)}
	\mathcal{T}_K:=\operatorname{Hom}_{C_2}(J_K^{S_\infty},C_m),
\end{equation}
where $C_m$ is the $C_2$-module defined by $g_2\cdot g_m=g_m^{-1}$ where $g_2$ is the generator of $C_2$ and $g_m$ is the generator of $C_m$.
Here $J_K$ is the id{\`e}les group of $K$, and $S_{\infty}$ is the set of primes at infinity.
Recall from Definition~\ref{def: product of ramified primes for morphisms} that we have a counting function on $\mathcal{T}_K$.
In particular, if $\rho:J_K^{S_\infty}\to C_m$ could be lifted to a map $\tilde{\rho}:\operatorname{C}_K\to C_m$, then
\begin{equation*}
	P(\tilde{\rho}):=P(\rho).
\end{equation*}
And in this case if $\hat{L}/K$ is the class field of $\tilde{\rho}$, then we have
\begin{equation*}
	P(L)=P(\hat{L})=P(K)P(\tilde{\rho}).
\end{equation*}
And this explains why we call the counting function $P$ for the $C_2$-morphisms the product of ramified primes.
In this way, $\mathcal{T}_K$ is equipped with the counting function $P$.
Now let's prove the theorem for the case when $m$ is a power of an odd prime.
\begin{proof}[Proof of Theorem~\ref{thm: estimate of Dq}]
	For each rational prime $p$, let $\mathcal{T}_{K_p,\Omega}$ be the subset of $\operatorname{Hom}_{C_2}(\mathscr{O}_p^*,C_m)$ such that a local $C_2$-morphism $\rho_p$ is contained in $\mathcal{T}_{K_p,\Omega}$ if and only if for some prime $\mathfrak{p}\mid p$ we have $\rho_{\mathfrak{p}}(\zeta_{\mathfrak{p}})\in\Omega$.
	Since $\Omega$ is closed under invertible powering and conjugation, we see that the definition of $\mathcal{T}_{K_p,\Omega}$ does not depend on the choice of $\mathfrak{p}$.
	Let
	\begin{equation*}
		\mathcal{T}_{K_p,\Omega}^c:=\operatorname{Hom}_{C_2}(\mathscr{O}_p^*,C_m)-\mathcal{T}_{K_p,\Omega}
	\end{equation*}
	be the complement.
	For each non-negative integer $r$, let $\mathcal{T}_{K,\Omega}^r$ be the subset of $\mathcal{T}_K$ such that a $C_2$-morphism $\rho$ is contained in $\mathcal{T}_{K,\Omega}^r$ if and only if
	\begin{equation*}
		\#\{p\nmid 2m\mid \rho_p\in\mathcal{T}_{K_p,\Omega}\}=r.
	\end{equation*}
	Clearly if $L\in\mathcal{S}_{\Omega}^r$ and $K$ is associated to $L$ in the sense that there exists a surjective $C_2$-morphism $\chi:\operatorname{C}_K\to C_m$ corresponding to the Galois closure $\hat{L}$ of $L/\mathbb{Q}$, then $\rho:=\prod_{p}\chi_p$ is contained in $\mathcal{T}_{K_p,\Omega}$.
	Let
	\begin{equation}\label{eqn: estimate of Dq 0}
		E_{\mathcal{T}_{K,\Omega}^r,P}(s):=\prod_{p}\bigl(\sum_{\rho\in\mathcal{T}^c_{K_p,\Omega}}p^{-c(\rho)s}\bigr)
		\sum_{
			\substack{
				n\text{ square-free}\\
				\gcd(m,n)=1,\omega(n)=r
			}
		}\prod_{p\mid n}\bigl(\sum_{\rho\in\mathcal{T}_{K_p,\Omega}}p^{-c(\rho)s}\bigr)
	\end{equation}
	be the series defined as in (\ref{eqn: estimate for F(r,s)}).
	Then Proposition~\ref{prop: counting homomorphisms} first implies that if we write $E_{\mathcal{T}_{K,\Omega}^r,P}(s)=\sum_{n=1}^{\infty}b_{K,n}n^{-s}$ in the form of a Dirichlet series, then
	\begin{equation}\label{eqn: estimate of Dq 1}
		N_{\mathcal{T}_{K,\Omega}^r,P}(X)\leq\sum_{n<X}b_{K,n}.
	\end{equation}
	Let
	\begin{equation*}
		N:=\max_{[K:\mathbb{Q}]=2}\max_{p}\{\lvert\operatorname{Hom}_{C_2}(\mathscr{O}_p^*,C_m)\rvert\},
	\end{equation*}
	where $K$ runs over quadratic number fields and $p$ runs over all rational primes (including infinity).
	The number $N$ exists because if $p\nmid 2m$, then $\lvert\operatorname{Hom}_{C_2}(\mathscr{O}_p^*,G)\rvert\leq m$.
	So the maximum is taken from finitely many local algebras.
	Note also that if $p$ is a finite prime, then 
	\begin{equation*}
		\#\{\rho\in\mathcal{T}_{K_p,\Omega}^c\mid c(\rho)=1\}=0,
	\end{equation*} 
	because $\Omega=C_m\backslash\{1\}$.
	Define
	\begin{equation}\label{eqn: estimate of Dq 2}
		E_{\mathcal{T}_{\Omega}^r,P}(s):=
		\prod_{p\mid m\infty}(1+Np^{-s})
		\sum_{
			\substack{
				n\text{ square-free}\\
				\omega(n)=r
			}
		}\prod_{p\mid n}(1+Np^{-s}),
	\end{equation}
	and write $E_{\mathcal{T}_{\Omega}^r,P}(s)=\sum_{n=1}^\infty a_nn^{-s}$.
	By comparing the local factors of (\ref{eqn: estimate of Dq 0}) and (\ref{eqn: estimate of Dq 2}), for each quadratic number field $K$ and for each $n\geq1$, we have
	\begin{equation}\label{eqn: estimate of Dq 3}
		a_n\geq b_{K,n}.
	\end{equation}
	On the other hand, for a continuous homomorphism $\rho:J_K^{S_\infty}\to C_m$, there are at most $\lvert\operatorname{Hom}(\operatorname{Cl}_K,C_m)\rvert$ different lifting to a map $\tilde{\rho}:\operatorname{C}_K\to C_m$.
	Using Corollary~\ref{cor: dihedral extensions and morphisms}, we see that
	\begin{equation}\label{eqn: dihedral relation between S and T}
		N_{\mathcal{S}_{\Omega}^r,P}(X)\leq 
		\sum_{
			\substack{
				[K:\mathbb{Q}]=2\\
				P(K)<X
			}
		}
		\lvert\operatorname{Hom}(\operatorname{Cl}_K,C_m)\rvert
		N_{\mathcal{T}_{K,\Omega}^r,P}(X/P(K)).
	\end{equation}
	Together with (\ref{eqn: estimate of Dq 1}) and (\ref{eqn: estimate of Dq 3}), we see that
	\begin{equation*}
		N_{\mathcal{S}_{\Omega}^r,P}(X)\leq 
		\sum_{
			\substack{
				[K:\mathbb{Q}]=2\\
				P(K)<X
			}
		}
		\lvert\operatorname{Hom}(\operatorname{Cl}_K,C_m)\rvert
		\sum_{n<X/P(K)}a_n.
	\end{equation*}
	Note that $E_{\mathcal{T}_{\Omega}^r,P}(s)$ is independent of the choice of $K$.
	Let
	\begin{equation*}
		H(X):=\sum_{
			\substack{
				[K:\mathbb{Q}]=2\\
				P(K)<X
			}
		}\lvert\operatorname{Hom}(\operatorname{Cl}_K,C_m)\rvert
		\quad\text{and}\quad
		A(X):=\sum_{n<X}a_n.
	\end{equation*}
	The asymptotic behaviour of $A(X)$ could be deduced from Proposition~\ref{prop: estimate of counting homomorphisms}.
	Recall that we have assumed that the Cohen-Lenstra Heuristics is true, or in this case, we take the following result as our condition: 
	as $X\to\infty$, we have
	\begin{equation*}
		H(X)\ll X.
	\end{equation*}	
	In terms of Riemann-Stieltjes integral (see Montgomery and Vaughan~\cite[Appendix A]{montgomery2006multiplicative} for example), as $X\to\infty$, we have
	\begin{equation*}
		\begin{aligned}
			N_{\mathcal{S}_{\Omega}^r,P}(X)
			\ll&\sum_{n<X}a_n\frac{X}{n}
			\\
			=&\int_{1^+}^X\frac{X}{t}\operatorname{d}A(t)
			\\
			=&\frac{X}{t}A(t)\Big\vert_{1^+}^X+X\int_{1^+}^X\frac{A(t)}{t^2}\operatorname{d}t
			\\
			\ll&X(\log\log X)^{r+1}.
		\end{aligned}
	\end{equation*}
	And the result just follows.
\end{proof}
Then we deal with Theorem~\ref{thm: D4 counting fields P}.
We use the similar method as above to prove it.
\begin{proof}[Proof of Theorem~\ref{thm: D4 counting fields P}]
	The proof is similar to that of Theorem~\ref{thm: estimate of Dq}, except that in this case we have the following
	\begin{equation*}
		\sum_{
			\substack{
				K=\mathbb{Q}(\sqrt{d})\\
				P(K)<X
			}
		}\lvert\operatorname{Hom}(\operatorname{Cl}_K,C_4)\rvert\sim c_2X\log X,
	\end{equation*}
	as $X\to\infty$,
	where $c_2$ is a positive constant.
	In other words, ``most'' homomorphisms are not surjective, and the main term is determined by the $2$-rank of $\operatorname{Cl}_K$.
	See Alex Smith~\cite{smith2022distribution} for details on the distribution of $\operatorname{Cl}_K[2^\infty]$ when $K$ runs over quadratic number fields.
	Again, using the notations of $\mathcal{T}_{K_p,\Omega}$ and $\mathcal{T}_{K_p,\Omega}^c$ and $\mathcal{T}_{K,\Omega}^r$ from the proof of Theorem~\ref{thm: estimate of Dq}, one can define the series $E_{\mathcal{T}_{K,\Omega}^r,P}(s)$ as in (\ref{eqn: estimate of Dq 0}).
	Write it as in the form of a Dirichlet series $E_{\mathcal{T}_{K,\Omega}^r,P}(s)=\sum_{n=1}^{\infty}b_{K,n}n^{-s}$.
	Note that the inequality
	\begin{equation*}
		N_{\mathcal{S}_{\Omega}^r,P}(X)
		\leq\sum_{
			\substack{
				[K:\mathbb{Q}]=2\\
				P(K)<X
			}
		}\lvert\operatorname{Hom}(\operatorname{Cl}_K,C_m)\rvert
		\sum_{n<X}b_{K,n}
	\end{equation*}
	is still true in this case by the proof of Theorem~\ref{thm: estimate of Dq}.
	Define in this case the series
	\begin{equation*}
		E_{\mathcal{T}_{\Omega}^r,P}(s)=
		\prod_{p\mid 2\infty}(1+Np^{-s})
		\prod_{p\nmid 2\infty}(1+p^{-s})
		\sum_{
			\substack{
				n\text{ square-free}\\
				\omega(n)=r
			}
		}\prod_{p\mid n}(1+Np^{-s}),	
	\end{equation*}
	where $N$ is a large enough number such that 
	\begin{equation*}
		N\geq\max_{[K:\mathbb{Q}]=2}\max_{p}\{\#\operatorname{Hom}_{C_2}(\mathscr{O}_p^*,C_4)\}.
	\end{equation*}
	Write it also in the form of a Dirichlet series: $E_{\mathcal{T}_{\Omega}^r,P}(s)=\sum_{n=1}^{\infty}a_nn^{-s}$.
	The comparison between $E_{\mathcal{T}_{K,\Omega}^r,P}(s)$ and $E_{\mathcal{T}_{\Omega}^r,P}(s)$ shows that for each quadratic number field and for each $n\geq1$, we have that
	\begin{equation*}
		a_n\geq b_{K,n}.
	\end{equation*}
	So we have that
	\begin{equation*}
		N_{\mathcal{S}_{\Omega}^r,P}(X)
		\leq\sum_{
			\substack{
				[K:\mathbb{Q}]=2\\
				P(K)<X
			}
		}\lvert\operatorname{Hom}(\operatorname{Cl}_K,C_m)\rvert
		\sum_{n<X/P(K)}a_n.
	\end{equation*}
	Note that $E_{\mathcal{T}_{\Omega}^r,P}(s)=\sum_{n=1}^{\infty}a_nn^{-s}$ is independent of the choice of $K$.
	And the asymptotic behaviour of $\sum_{n<X}a_n$ could be deduced from Proposition~\ref{prop: estimate of counting homomorphisms}.
	So it reduces to the following computation.
	Write 
	\begin{equation*}
		A(X):=\sum_{n<X}a_n
		\quad\text{and}\quad 
		H(X):=\sum_{
			\substack{
				[K:\mathbb{Q}]=2\\
				P(K)<X
			}
		}\lvert\operatorname{Hom}(\operatorname{Cl}_K,C_4)\rvert.
	\end{equation*}
	Let $M$ be a fixed constant such that for all $X>M$ we have
	\begin{equation*}
		H(X)\leq 2c_2X\log X\quad\text{and}\quad A(X)\leq d_rX(\log\log X)^r,
	\end{equation*}
	where $d_r$ is a positive constant.
	The existence of such a number $M$ follows from our assumption on $H(X)$ and the asymptotic behaviour of $A(X)$.
	Then for large $X$, we have
	\begin{equation*}
		\begin{aligned}
			&\sum_{
				\substack{
					[K:\mathbb{Q}]\\
					P(K)<X
				}
			}\lvert\operatorname{Hom}(\operatorname{Cl}_K,C_4)\rvert
			\sum_{n<X/P(K)}a_n
			\\
			=&\sum_{n<X}a_n
			\sum_{
				\substack{
					[K:\mathbb{Q}]=2\\
					P(K)<X/n
				}
			}\lvert\operatorname{Hom}(\operatorname{Cl}_K,C_4)\rvert
			\\
			\ll&\sum_{X/M\leq n<X}a_n H(X/n)+\sum_{n<X/M}a_n\cdot\frac{X}{n}\log X
			\\
			\ll&\sum_{n<X}a_n+\int_{1^+}^{X/M}\frac{X}{t}\log X\operatorname{d}A(t)
			\\
			\ll&X(\log\log X)^r+\frac{X}{t}\log X A(t)\Big\vert_{1^+}^{X/M}+X\log X\int_{1^+}^{X/M}\frac{A(t)}{t^2}\operatorname{d}t
			\\
			\ll&X\log X(\log\log X)^r+X\log X\cdot\Biggl[\Bigl(\int_{1^+}^M+\int_{M}^{X/M}\Bigr)\frac{A(t)}{t^2}\operatorname{d}t\Biggr]
			\\
			\ll&X(\log X)^2(\log\log X)^{r+1}.
		\end{aligned}
	\end{equation*}
\end{proof}
\begin{remark}
	Note that the computation under some heuristics including those of Cohen-Lenstra and Malle-Bhargava shows that when ordered by the conductor $C$ introduced in Section~\ref{section:D4 extensions} for $D_4$-fields, it is not likely that we can prove a result of zero-probability of $\operatorname{rk}_2 2\operatorname{Cl}_L\leq r$.
	Though even the heuristics hold in this case, the theories we've introduced here can only imply that
	\begin{equation*}
		\mathbb{E}_{\mathcal{S},C}(\lvert\operatorname{Hom}(2\operatorname{D}^{D_4}_L,C_2)\rvert)<\infty,
	\end{equation*}
	where $\operatorname{D}_L^{D_4}\subseteq\operatorname{Cl}_L$ (see also Definition~\ref{def:invariant part of class groups}).
	Still results of this section indicates that ordering fields by the product of ramified primes keeps some possibility for us to prove the statistical results related to the infinite $C_2$-moments on the higher torsion part of the class group.
\end{remark}

\section*{Acknowledgement}
I would like to thank Melanie Wood for many useful conversations related to this work.
I would like to thank Yuan Liu for comments that helped improve the study such that the related results could be generalized to relative class groups.
I would like to thank Alex Bartel, Aaron Landesman, Fabian Gundlach, and Michael Kural for many useful discussions related to this paper.

\section*{Conflict of Interest Statement}
The author was supported by National Science Foundation grants DMS-2052036 and DMS-2140043 when the related work started, and now by Jiangsu Funding Program for Excellent Postdoctoral Talent 2023ZB064.
The author declare that there is no potential conflict.

\section*{Data Availability Statement}
The author claims that the data supporting the study of this paper are available within the article.

\bibliographystyle{plain}
\bibliography{references}

\begin{thebibliography}{10}

\bibitem{altug2017number}
Salim~Ali Altug, A.~I.~Vijaya Shankar, Ila Varma, and Kevin~H. Wilson.
\newblock The number of $d_4$-fields ordered by conductor.
\newblock {\em Journal of the European Mathematical Society}, 23:2733--2785,
  2021.

\bibitem{bartel2020class}
Alex Bartel and Hendrik~W Lenstra~Jr.
\newblock On class groups of random number fields.
\newblock {\em Proceedings of the London Mathematical Society},
  121(4):927--953, 2020.

\bibitem{bhargava07mass}
Manjul Bhargava.
\newblock Mass formulae for extensions of local fields, and conjectures on the
  density of number field discriminants.
\newblock {\em International Mathematics Research Notices}, 2007, 2007.

\bibitem{bhargava2013davenport}
Manjul Bhargava, Arul Shankar, and Jacob Tsimerman.
\newblock On the davenport--heilbronn theorems and second order terms.
\newblock {\em Inventiones mathematicae}, 193(2):439--499, 2013.

\bibitem{CM90}
Henri Cohen and Jacques Martinet.
\newblock Étude heuristique des groupes de classes des corps de nombres.
\newblock {\em Journal für die reine und angewandte Mathematik}, 404:39--76,
  1990.

\bibitem{Coleman_1990}
M.~D. Coleman.
\newblock A zero-free region for the hecke l-functions.
\newblock {\em Mathematika}, 37(2):287–304, 1990.

\bibitem{Davenport1971Cubic}
H.~{Davenport} and H.~{Heilbronn}.
\newblock {On the Density of Discriminants of Cubic Fields. II}.
\newblock {\em Proceedings of the Royal Society of London Series A},
  322(1551):405--420, May 1971.

\bibitem{Delange54}
Hubert Delange.
\newblock G\'en\'eralisation du th\'eor\`eme de {Ikehara}.
\newblock {\em Annales scientifiques de l'\'Ecole Normale Sup\'erieure}, 3e
  s{\'e}rie, 71(3):213--242, 1954.

\bibitem{directions2016}
Ellen Eischen, Ling Long, Rachel Pries, and Katherine Stange.
\newblock {\em Directions in Number Theory: Proceedings of the 2014 WIN3
  Workshop}, volume~3.
\newblock Springer, Cham, 01 2016.

\bibitem{Gerth1987}
III Frank~Gerth.
\newblock Extension of conjectures of cohen and lenstra.
\newblock {\em Expositiones Mathematicae}, 5(2):181 -- 184, 1987.

\bibitem{ishida1976genus}
Makoto Ishida.
\newblock {\em The genus fields of algebraic number fields}.
\newblock Lecture notes in mathematics. Springer, Berlin, Heidelberg, 1
  edition, 1976.

\bibitem{iwaniec2004analytic}
Henryk Iwaniec and Emmanuel Kowalski.
\newblock {\em Analytic Number Theory}.
\newblock American Mathematical Society Colloquium Publications. American
  Mathematical Society, Providence, Rhode Island, 2004.

\bibitem{malle2004distribution}
Gunter Malle.
\newblock On the distribution of galois groups, ii.
\newblock {\em Experimental Mathematics}, 13(2):129--135, 2004.

\bibitem{montgomery2006multiplicative}
H.L. Montgomery and R.C. Vaughan.
\newblock {\em Multiplicative Number Theory I: Classical Theory}.
\newblock Cambridge Studies in Advanced Mathematics. Cambridge University
  Press, Cambridge, 2006.

\bibitem{narkiewicz2014elementary}
W.~Narkiewicz.
\newblock {\em Elementary and Analytic Theory of Algebraic Numbers}.
\newblock Springer, Berlin Heidelberg, 2014.

\bibitem{neukirch2013algebraic}
J.~Neukirch and N.~Schappacher.
\newblock {\em Algebraic Number Theory}.
\newblock Grundlehren der mathematischen Wissenschaften. Springer, Berlin
  Heidelberg, 2013.

\bibitem{neukirch2013cohomology}
J.~Neukirch, A.~Schmidt, and K.~Wingberg.
\newblock {\em Cohomology of Number Fields}.
\newblock Grundlehren der mathematischen Wissenschaften. Springer, Berlin
  Heidelberg, 2013.

\bibitem{RZ1969ClassRank}
P.~Roquette and H.~Zassenhaus.
\newblock A class rank estimate for algebraic number fields.
\newblock {\em Journal of the London Mathematical Society}, s1-44(1):31--38,
  1969.

\bibitem{smith2022distribution}
Alexander {Smith}.
\newblock {The distribution of $\ell^\infty$-Selmer groups in degree $\ell$
  twist families}.
\newblock {\em arXiv e-prints}, page arXiv:2207.05674, July 2022.

\bibitem{wang2021moments}
Weitong Wang and Melanie~Matchett Wood.
\newblock Moments and interpretations of the cohen--lenstra--martinet
  heuristics.
\newblock {\em Commentarii Mathematici Helvetici}, 96(2):339--387, 2021.

\bibitem{wood2010probabilities}
Melanie~Matchett Wood.
\newblock On the probabilities of local behaviors in abelian field extensions.
\newblock {\em Compositio Mathematica}, 146(1):102--128, 2010.

\end{thebibliography}


\end{document}